\documentclass[reqno]{amsart}
\usepackage{amsmath}
\usepackage[nobysame]{amsrefs}
\usepackage{array}
\usepackage{amsfonts}
\usepackage{url} 
\usepackage{xcolor} 
\usepackage{fancyhdr}
\usepackage{multirow}
\usepackage{graphicx}
\usepackage{relsize}
\usepackage{amssymb}  
\usepackage{hyperref}
\usepackage{mathtools}
\usepackage{indentfirst}
\usepackage{comment}
\usepackage{enumerate}
\usepackage{float}
\usepackage[normalem]{ulem}

\newtheorem{theorem}{Theorem}[section]
\newtheorem*{question*}{Question}
\newtheorem{question}[theorem]{Question}
\newtheorem{lemma}[theorem]{Lemma}
\newtheorem{corollary}[theorem]{Corollary}
\newtheorem{remark}[theorem]{Remark}
\newtheorem{example}[theorem]{Example}

\newtheorem{definition}[theorem]{Definition}

\hypersetup{colorlinks=true,linkcolor=blue, linktocpage}
\hypersetup{colorlinks,
citecolor=blue,
allcolors=blue
}
\DeclareMathOperator{\dist}{\rho}
\DeclareMathOperator{\diam}{diam}
\DeclareMathOperator{\Int}{int}
\DeclareMathOperator{\Rec}{Rec}
\DeclareMathOperator{\Per}{Per}
\DeclareMathOperator{\End}{End}
\DeclareMathOperator{\Cut}{Cut}
\DeclareMathOperator{\C}{C}
\DeclareMathOperator{\lcm}{lcm}
\DeclareMathOperator{\Fix}{Fix}
\DeclareMathOperator{\Br}{Br}
\DeclareMathOperator{\V}{V}

\title[Limit sets in hyperspace]{On limit sets and equicontinuity in the hyperspace of continua in dimension one}
\author[D. Jeli\'c]{Domagoj Jeli\'c}
\address[D. Jeli\'c]
{Faculty of Science, University of Split, Ru\dj era Bo\v{s}kovi\'ca 33, 21000 Split, Croatia
}
\author[P. Oprocha]{Piotr Oprocha}
\address[P. Oprocha]{
	Centre of Excellence IT4Innovations - Institute for Research and Applications of Fuzzy Modeling, University of Ostrava, 30. dubna 22, 701 03 Ostrava 1, Czech Republic.
}
\date{\today}
\begin{document}

\begin{abstract}
	The paper studies the structure of $\omega$-limit sets of map $\tilde{f}$ induced on the hyperspace $C(G)$ of all connected compact sets, by
	dynamical system $(G,f)$ acting on a topological graph $G$. In the case of the base space being a topological tree we additionally show that $\tilde{f}$ is always almost
	equicontinuous and characterize its Birkhoff center.
%
%
\end{abstract}

\maketitle

\section{Introduction}
Qualitative theory of dynamical systems tries to understand long time behavior of the system under observation. 
In 
this framework, it is interesting to understand how the majority of points behave, or how collective dynamics behaves. From statistical point of view it can be 
achieved by understanding evolution of densities or probability measures. From topological perspective we can try to look at evolution of certain sets. A natural way of doing this
is extending given continuous selfmap $f$ of a compact metric space $X$,
to the so-called \emph{induced} mappings $\overline{f}$ and $\tilde{f}$ on the hyperspace $2^X$ of compact subsets of $X$ and the hyperspace $C(X)$ of continua in $X$, respectively, both defined in an obvious way.  
These induced maps, if we equip the hyperspace $2^X$ with \emph{Hausdorff metric}, become dynamical systems themselves.
The main question asks how much richer is dynamics of these maps, and how much dynamics governed by evolution of sets is different compared to individual evolution of points.

 The first study of the relation between 
 these two realms dates back to 1975 and was done by Bauer and Sigmund in~\cite{Bauer}.
 They proved that while evolution of probability measures is to some extent similar to base map $f$, the dynamics of $\overline{f}$ usually leads to much more complicated dynamics. 
{Since then, many authors were interested in various questions regarding this topic.
	Those include, although are not limited to, the notions such as: transitivity, mixing and weakly mixing~\cite{Banks, KOCSF, Dendrite1} where it was for instance proved that in the case of a graph or dendrite base map, the induced map on continua is never transitive. 
	Furthermore, the shadowing of the induced map was also studied lately~\cite{Fernandez, Arbieto}.
	The most recently, the authors in~\cite{Djoric} calculated polynomial entropy of induced circle and interval homeomorphism, while the authors of~\cite{Ju}, among the other things, studied the structure of $\omega$-limit sets of induced map on the hyperspace of circle subcontinua.
	Let us emphasize that the latter work is generalized by the results of this paper. 
	The topological entropy, as one of the most explored indicators of chaotical behavior of dynamical systems, was also studied in this context 
 (e.g. see~\cite{KOCSF, Banks, Lampart}) and it turned out that complexity of dynamics on $\C(X)$, which is relatively small subset of $2^X$,
 may grow very much when compared with $f$. In fact, the only known cases when entropy of $\tilde f$ remains the same as of $f$ are one dimensional spaces such as interval \cite{MatviiInterval}, or more generally topological graphs \cite{Jelic}.
  Dynamics on more complicated one-dimensional spaces, such as dendrites (see~\cite{Dendrite1} or~\cite{Dendrite2}) is not that tightly bound to dynamics of $f$ and can reveal much richer orbits structure.
 	
  In the present paper we go beyond results of \cite{Jelic}, providing a characterization of possible $\omega$-limit sets, topological center and equicontinuity points in the induced system $\tilde f$  when the underlying space is a topological graph or a topological tree. }
The first main result of this paper is the following theorem.

\begin{theorem}\label{thm:main}
Let $G$ be a topological graph and let $\C(G)$ be the hyperspace of all subcontinua of $G$.
	For any $A\in\C(G)$  (at least) one of the following properties holds:
\begin{enumerate}[(i)]
\item\label{asy_per} $A$ is asymptotically periodic,
\item\label{wandering} $A$ is wandering, i.e. $f^k(A)\cap f^j(A)=\emptyset$ for all $j\neq k$, or
\item\label{circ} there is some circumferential set $M=E(K,f)$ and some $k\geq 0$ such that $f^k(A)\subset K$.
\end{enumerate}
\end{theorem}
	
	This result generalizes the main results of~\cite{Matviichuk} and~\cite{Jelic}.
	Namely, in~\cite{Matviichuk}, Matviichuk proved that any subcontinuum of a topological tree is either asymptotically periodic or asymptotically degenerate (or both).
	On the other hand, in~\cite{Jelic} the authors proved that, in the case of a topological graph, the entropy of induced system of subcontinua is equal to the entropy of the base system and 
	provided a full characterization of the recurrent continua.
	These results are a starting point for our approach to proving the Theorem~\ref{thm:main}, which then will be used as a tool in further analysis of dynamics of $\tilde f$.
	
	In the case of a topological tree, we will describe the center of $\tilde f$, that is, the closure of the set of recurrent subtrees, and, furthermore, show almost equicontinuity of the induced system.
	
	The former of two results is a generalization of a theorem from~\cite{MatviiCenter} where Matviichuk proved that, in the case of a compact interval, every subinterval 
	fully covered by the orbit of a periodic interval is in fact either 
	\begin{enumerate}
\item periodic subinterval,
\item a singleton formed by a point from the closure of the set of periodic points of the base map,
\item is asymptotically degenerate, nondegenerate subinterval.	
	\end{enumerate}
{Inspired by the above, we provide in Theorem~\ref{thm:center-charact} a full characterization of the Birkhoff center for induced hyperspace map on topological trees.
		In the theorem, we describe those nondegenerate trees belonging to the center which are not periodic. Roughly speaking, they are wandering intervals attracted by solenoids.}
%
%
	
Finally, in the case of a topological tree, we prove somehow surprising result, that the set of points of equicontinuity is dense $G_\delta$ in the hyperspace of subcontinua, that is:

\begin{theorem}\label{thm:almost_equi}
Let $f\colon T\to T$ be a continuous selfmap of a tree.
Then $(C(T),\tilde{f})$ is almost equicontinuous.
\end{theorem} 
It extends the main theorem of~\cite{Rybak} where Rybak proved that, for a compact interval, its induced system on subintervals contains a Lyapunov-stable point, which in modern terminology means that there is a point of equicontinuity. 
Note that the equicontinuity in \eqref{thm:almost_equi} is completely independent of dynamics of $f$. In particular, we can start with a map with the specification property, high entropy, etc.
But then, when passing to the hyperspace $\C(T)$, large structure is built over original space $T$, where dynamics behaves in some kind of regularity. It to some extend explains why the entropy cannot ``explode'' in these cases.

The proofs of Theorem~\ref{thm:center} and Theorem~\ref{thm:almost_equi}, in huge part rely on tools provided by Theorem~\ref{thm:main} and celebrated ``spectral decomposition'' theorem of Blokh \cite{Blokh}, but an additional, important ingredient is a proper characterization of the set of periods of trees periodic in the induced system and containing a {fixed point of the base map}.
	In what follows we will prove that if $f\colon T\to T$ is a tree map and $A$ is a subtree of $T$ which is a periodic point for $\tilde{f}$ of some minimal period $p$ and
if $A$ contains a fixed point $x\in T$ of $f$, then $p$ cannot be greater than some bound $M$, where this number solely depends on the number of endpoints of $T$.
 This is a generalization of the result of Fedorenko~\cite{Fedorenko}, where it was proved that this bound equals $2$ when the underlying space of the base system is a compact interval.
In case of topological graphs there is no clear candidate that can be effectively used in the place of endpoints of the tree. Because of this difficulty, we were not able to prove the result
in general topological graphs, and so leave it as a problem for future research:
\begin{question}
	Are the statements similar to those of Theorems~\ref{thm:periods}, \ref{thm:center-charact} and~\ref{thm:almost_equi} also true for graph maps?
\end{question}
	
The paper is organized as follows: In Section~\ref{sec:preliminaries} we recall basic definitions and some properties of graph maps and hyperspace dynamics.
	Section~\ref{sec:proof} brings proof of the main Theorem~\ref{thm:main}.
	Section~\ref{sec:periods} is devoted to finding the bound on the periods of periodic points of the induced map in the case of topological trees, which is used later in sections~\ref{sec:center} and~\ref{sec:equicontinuity} to characterize center of the induced map in the tree case as well as to prove its almost equicontinuity.

\section{Preliminaries}\label{sec:preliminaries}
\subsection{Topological graphs}
	A \emph{topological graph} (or short \emph{graph}) is a 
	continuum $G$ 
	 such that there is a one dimensional simplical complex $K$ whose geometric carrier $|K|$ is homeomorphic to $G$.
	 If $G$ is a singleton then we call it a degenerate graph.
	
	If $G$ is a graph, then there is a finite set $\V(G)\subset G$ such that each connected component of $G\setminus \V(G)$ is homeomorphic to an open interval.
	We can assume that $\V(G)$ is minimal subset of $G$ with this property, i.e. that for each $v\in \V(G)$ and each connected neighborhood $U$ of $v$ in $G$, such that $U\cap \V(G)=\{v\},$
	$U\setminus \{v\}$ has either one or at least three components.  
	Having said that, each element of $\V(G)$ will be called \emph{vertex} and
	closure in $G$ of each component of $G\setminus \V(G)$ will be called an \emph{edge} of $G$.
	Notice that each edge is either homeomorphic to a closed interval or to the circle $\mathbb{S}^1$.
	Edges homeomorphic to $\mathbb{S}^1$ will be called \emph{loops} of $G$.
	A \emph{branching point} of $G$ is a point in $G$ having no neighborhood homeomorphic to an interval.
	The set of all branching points of $G$ is denoted by $\Br(G)$.
	It is clear that each branching point is an element of $\V(G)$.
	Elements of $\V(G)$ different from branching points will be called \emph{endpoints} of $G$.
	We denote the set of all endpoints of $G$ by $\End(G)$.
We denote the set of all endpoints of $G$ by $\End(G)$.
For every point $x\in G$, we define the \emph{valence of $x$} as the number of connected components of the set $U\setminus \{x\}$, where $U$ is any neighborhood of $x$ in $G$ not containing branching points nor closed curves.
 Therefore, the sets $\Br(G)$ and $\End(G)$ consist exactly of those points of $G$ whose valences are $1$ and at least $3$, respectively, while the valences of all the other points of $G$ equal $2$. 
	Moreover, every nondegenerate subcontinuum of $G$ will be called a \emph{subgraph}.
	A \emph{topological tree} (or short \emph{tree}) is a graph not containing subsets homeomorphic to $\mathbb{S}^1$.
	In this context, a tree which is singleton will be called a \emph{degenerate tree}.
	Note that for each topological tree $T$ and each point $x\in T\setminus \End(T)$, the set $T\setminus \{x\}$ is not connected.
	Therefore, the points of a tree which are not its endpoints are called the \emph{cut points}.
	The set of all cut points points of a tree $T$ is denoted by $\Cut(T)$.

\begin{remark}\label{rem:ambient}
	From this moment on we identify $G$ with a geometric carrier of simplicial complex defining it.
	Strictly speaking, we assume that graph $G$ is a subset of $\mathbb{R}^3,$ its edges lie on straight lines
	and on each edge we have a metric assigning to it the length one.
		Moreover, the metric $d$ on $G$ now can be chosen in a way that $d(x,y)$ equals the length of the shortest path between $x$ and $y.$
		Induced topology matches the original subspace topology of $G$.
	\end{remark}

\subsection{Dynamical systems}
	A (discrete, topological) \emph{dynamical system} is a pair $(X,f)$ where {$(X,\dist)$} is a compact metric space and $f\colon X\to X$ is a continuous map.
	We will often identify dynamical system $(X,f)$ with the map $f$.
	Any continuous map $f\colon G\to G$, where $G$ is a graph, will simply be called a \emph{graph map}.
	For a point $x\in X$, we define its \emph{orbit} as the set $\mathcal{O}_f(x)=\{f^n(x)\colon n\geq 0\}$ and its \emph{trajectory} as the sequence $\left(f^n(x)\right)_{n\geq 0}$.
{Similarly, the \emph{orbit of a set} $A\subset X$ is the set $\mathcal{O}_f(A)=\bigcup_{n=0}^{\infty}f^n(A).$
	When the map $f$ is clear from the context, we simply write $\mathcal{O}(x)$.}
	A point $x\in X$ is said to be \emph{periodic} if there is some $p>0$ such that $f^p(x)=x.$
	The smallest such $p$ is called the \emph{period} of $x$.
	We denote the set of all periodic points of $f$ by $\Per(f)$.
	A point $x\in X$ is said to be \emph{eventually periodic} if there is some $n\geq 0$ such that $f^n(x)\in\Per(f).$ 
	Equivalently, a point is eventually periodic if and only if it has finite orbit.

	If $(X,f)$ and $(Y,g)$ are two dynamical systems, a \emph{semi-conjugacy} between $f$ and $g$ is any surjective continuous map $\phi\colon X\to Y$ such that $\phi\circ f=g\circ \phi$.
	If in addition $\phi$ is bijective, and hence a homeomorphism, then it is a \emph{conjugacy} between $f$ and $g$ and we say that these two dynamical systems are \emph{conjugate}.

	A set $A\subset X$ is said to be {\emph{$f$-invariant}} if $f(A)\subset A$ and {\emph{strongly $f$-invariant}} if $f(A)=A.$
	A closed invariant set $M\subset X$ without closed and invariant proper subset is called \emph{minimal}.
	It is well known that closed invariant set $M\subset X$ is minimal if and only if the orbit of each point $x\in M$ is dense in $M$.	
	If $X$ is minimal, we say that dynamical system $(X,f)$ is minimal or, simply,  
	that $f$ is \emph{minimal map}.
	If $f^n$ is minimal for all $n>0$, $f$ is said to be \emph{totally minimal}.
	
	Map $f\colon X\to X$ is \emph{transitive} if for every pair of nonempty open subsets $U, V\subset X$ there is some $n>0$ such that $f^n(U)\cap V\neq\emptyset$.
	If $X$ does not contain isolated points, then $f$ is transitive if and only if there is $x\in X$ whose orbit is dense in $X$.	
	A map $f\colon X\to X$ is \emph{mixing} if for every pair of nonempty open subsets $U, V\subset X$ there is an $N\geq 0$ such that $f^n(U)\cap V\neq \emptyset$ for all $n\geq N$.	
	It is clear that each mixing map is transitive but not vice-versa.
	For a point $x\in X$, the limit set of its trajectoy is called \emph{$\omega$-limit set of $x$} and denoted by $\omega_f(x)$.
	In other words, $y\in\omega_f(x)$ if and only if there is a strictly increasing sequence of positive integers $(n_i)_{i\geq 0}$ such that $\left(f^{n_i}(x)\right)\to y$ as $i\to\infty$.
	A point $x\in X$ is \emph{non-wandering} if for every neighborhood $U$ of $x$ and every $N>0$ there is some $n>N$ such that $f^n(U)\cap U\neq\emptyset.$	Otherwise, we call it a \emph{wandering} point.
	A point $x\in X$ is \emph{recurrent} if $x\in\omega_f(x)$. 
	The set of all recurrent points is denoted by $\Rec (f)$ and we denote $\omega(f)=\bigcup_{x\in X}\omega_f(x).$
	We also define $\C(f)=\overline{\Rec(f)}$ and call this set the \emph{(Birkhoff) center} of $(X,f)$.	
    {Finally, a point $x\in X$ is said to be \emph{asymptotically periodic} if $\omega_f(x)$ is a periodic orbit.}
	
	A \emph{backward branch} of a point $x\in X$ is any sequence $\{x_i\}_{i\leq 0}$ in $X$ such that $x_0=x$ and $f(x_i)=x_{i+1}$ for each $i<0$. 
	A point $y$ belongs to the \emph{$\alpha$-limit set of a backward branch $\{x_i\}_{i\leq 0}$}, denoted by $\alpha_f\left(\{x_i\}_{i\leq 0}\right)$, if and only if there is a strictly decreasing sequence of negative integers $\{n_i\}_{i\geq 0}$ such that $x_{n_i}\to y$ as $i\to \infty$.
	It is easy to
see that both $\omega$-limit sets and $\alpha$-limit sets of backward branches are closed strongly
invariant sets.

 A point $x\in X$ is said to be \emph{equicontinuous} if
	\[\forall\epsilon>0,\ \exists\delta>0,\ \forall y\in B(x,\delta),\ \forall n\geq 0,\ \dist(f^n(y),f^n(x))<\epsilon.\]
	Let us denote by $\mathcal{E}$ the set of all equicontinuous points of $(X,f)$.
	Under this notation, $(X,f)$ is said to be \emph{equicontinuous} if $\mathcal{E}=X$.
	Furthermore, we say that $(X,f)$ is \emph{almost equicontinuous} if $\mathcal{E}$ is a residual set, or equivalently, since $X$ is compact metric space, if $\mathcal{E}$ is dense $G_\delta$ set in $X$.

Two points $x_1,x_2\in X$ are said to be \emph{proximal} if there exists an increasing sequence $(n_k)$ of nonnegative integers such that $\lim_k\dist(f^{n_k}(x_1),f^{n_k}(x_2))=0$.
	A set $S\subset\mathbb{N}_0$ is said to be \emph{syndetic} if it can be arranged as an increasing sequence $s_1<s_2<...$ with uniformly bounded gaps $s_{n+1}-s_n$.
	A point $x\in X$ is said to be \emph{uniformly recurrent} if for all $\epsilon>0$ the set $\{n\geq 0\colon \dist(f^n(x),x)<\epsilon\}$ is syndetic.

\begin{definition}
	Let $f\colon X\to X$ and $g\colon Y\to Y$ be two continuous maps of compact metric spaces $X,Y$ and {$M\subset X$} be a closed invariant set.
	A semi-conjugation $\phi\colon X\to Y$ is an \emph{almost conjugacy between $f\vert_M$ and $g$} if
	\begin{enumerate}[(1)]
		\item\label{almost1} $\phi(M)=Y$,
		\item $\forall y\in Y$, $\phi^{-1}(y)$ is connected,
		\item $\forall y\in Y$, $\phi^{-1}(y)\cap M=\partial\phi^{-1}(y)$, where $\partial A$ denotes the boundary of $A$ in $X$,
		\item\label{alomst4} $\exists N\geq 1$ such that $\forall y\in Y$, $\phi^{-1}(y)\cap M$ has at most $N$ elements (and, by \eqref{almost1}, at least one element).
	\end{enumerate}
\end{definition}

When $X$ and $Y$ are graphs then the last condition is a consequence of the previous ones since 
both singletons and subgraphs of $G$ have boundaries of uniformly bounded cardinality.

{ A typical example of dynamical system semi-conjugated to irrational rotation is Denjoy map (see~\cite[Example~14.9]{Devaney}; we comment on it more later when introducing circumferential sets). Another simple example is the truncated tent map $f$ on $[0,1]$ given for some $s>2$ by 
$$
f(x)=\begin{cases}
sx &, x\in [0,1/s]\\
s(1-x) &, x\in [1-1/s,1]\\
1 &, \text{otherwise}
\end{cases}
$$ 
If we take map $\phi\colon [0,1]\to [0,1]$ which collapse the interval $J=[1/s,1-1/s]$ as well as any of the intervals in the preimage $f^{-n}(J)$ to a point $f$ will transform (after proper rescaling of the domain) to the standard tent map
$$
g(x)=\begin{cases}
2x &, x\in [0,1/2]\\
2(1-x) &, x\in [1/2,1]
\end{cases}
$$ 
and $\phi$ an almost conjugacy between $f\vert_M$ and $g$, where $M=[0,1]\setminus \bigcup_{n\geq 0}f^{-n}((1/s,1-1/s))$.

\begin{remark}
It is clear that if $f\vert_M$ and $g$ are almost conjugate by a map $\phi$, then $\phi|_M$ is a semi-conjugacy, but it is not an almost conjugacy except the case it is
conjugacy itself.
\end{remark}
}

\subsection{$\omega$-limit sets in graph maps}
	In~\cite{MaiClosed} the authors proved that the family of all $\omega$-limit sets of a graph map is closed with respect to the Hausdorff metric in the hyperspace of all closed subsets of the graph.
	As a direct consequence, $\omega$-limit set of any point in $G$ lies in a maximal $\omega$-limit set.
	

		The set $\mathcal{O}(K)=\bigcup_{i=0}^kf^i(K)$ is called a \emph{cycle of graphs of period $k$} if $K$ is a subgraph of $G$ such that $K,f(K),...,f^{k-1}(K)$ are pairwise disjoint and $f^k(K)=K.$  

 A \emph{generating sequence} or a \emph{sequence generating a solenoidal set} is any nested sequence of cycles of graphs $K_1\supset K_2\supset\cdots$ for $f$ with periods tending to infinity. By definition $Q=\bigcap_n K_n$ is closed and strongly invariant, i.e. $f(Q)=Q$. Any closed and strongly invariant subset of $Q$ is called a \emph{solenoidal set}. The characterization of Blokh (e.g. see \cite[Theorem 1]{B1}) shows that $Q$ contains a perfect minimal set $Q_{min}=Q\cap\overline{Per f}$ such that $Q_{min}=\omega(x)$, for all $x\in Q$, and a maximal $\omega$-limit set (with respect to inclusion) $Q_{max}$ such that $Q_{max}=Q\cap \omega(f)$ which we will refer to as \emph{maximal solenoid}. 

	If $x$ is a point of a graph $G$, then by a \emph{side $T$ of the point $x$} we mean a family of open, non-degenerate arcs $\{V_T (x)\}$ containing no branching points, with one endpoint at $x$, such that $\bigcap_{V\in T} \overline{V}=\{x\}$ and if $U,V\in T$, then either $V\subset U$ or $U\subset V$. Members of the family $T$ are called $T$-sided neighborhoods of $x$.\\
 Let $f\colon G\to G$ be a graph map and $K\subset G$ be a cycle of graphs. For every $x\in K$, we define the \emph{prolongation set of $x$ with respect to $f|_K$}: 
 \[P_K(x,f)=\bigcap_U\overline{\bigcup_{i=1}^{\infty}f^i(U)},\]
 where $U$ is a relative neighborhood of $x$ in $K$. If the acting map $f$ is clear from the context, we will simply write $P_K(x)$ and when $K=G$ then we will write $P(x,f)$ or simply $P(x)$.  Observe that just as $x$ being recurrent is equivalent to $x\in\omega_f(x)$, $x$ being non-wandering is equivalent to $x\in P(x)$. Obviously, $P(x)$ is an invariant closed set and the map $f|_{P(x)}$ is surjective whenever $x$ is a non-wandering point. Similarly, we define the \emph{prolongation set of $x$ with respect to a side $T$}:
 \[P^T_K(x,f)=\bigcap_{V_T(x)}\overline{\bigcup_{i=1}^{\infty}f^i(V_T(x))},\]
 where the intersection is taken over relative $T$-sided neighborhoods $V_T(x)$ of $x$ in $K$. 
 
 An arc $V\subseteq G$  is \emph{non-wandering} if there is an integer $m\geq 1$ such that $f^m(V)\cap V\neq\emptyset$. 
 If every set $V_T(x)\in T$ is non-wandering then $P^T(x)$ is one of the following (see \cite{Forys}){}:
 	\begin{itemize}
 		\item $P^T(x)$ is a periodic orbit,
 		\item $P^T(x)$ is a cycle of graphs,
 		\item $P^T(x)$ is a solenoidal set $Q$.
 	\end{itemize}
 	
{Let $K\subset G$ be a cycle of graphs.} We define the following sets:
\[E(K,f)=\{x\in K: P_K(x,f)=K\}\]
and
\[E_S(K,f)=\{x\in K: \text{there is a side $T$ such that } P^T_K(x,f)=K\}.\] 
Clearly, $E_S(K,f)\subseteq E(K,f)$. These sets are closed and invariant. If $E(K,f)$ is infinite then, by~\cite[Theorem 2]{B1}, $E_S(K,f)= E(K,f)$. In general, $E_S(K,f)\neq E(K,f)$ and $f(E(K, f))\neq E(K, f)$.
\begin{theorem}[Blokh, \cite{B1,B2}]\label{E}
	Let $K\subset G$ be a cycle of graphs such that $E_S(K,f)$ is non-empty. If $E_S(K,f)$ is finite then it is a periodic orbit. Otherwise, $E_S(K,f)= E(K,f)$ and it is an infinite maximal $\omega$-limit set.
\end{theorem}

Let $E(K,f)$ be an infinite maximal $\omega$-limit set from Theorem ~\ref{E}. We say that $E(K,f)$ is a \emph{basic set} if $\Per (f)\cap K\neq \emptyset$ and we denote it by $D(K,f)$, and in the opposite case when $\Per (f)\cap K= \emptyset$ we say that $E(K,f)$ is a \emph{circumferential set} and we denote it by $S(K,f)$. We will simply write $D(K)$ and $S(K)$ in the case where $f$ is clear from the context.

	For an infinite $\omega$-limit set $\omega_f(x)$, let
	\[\mathcal{C}(x)=\{X\colon X\subset G \text{ is a cycle of graphs and } \omega_f(x)\subset X\}.\]
	One can show that $\mathcal{C}(x)$ is never empty.
	Following ideas of Blokh in their paper \cite{Snoha}, Ruette and Snoha proved that given $x\in G$ and a cycle of graphs $K$, if
	$\omega_f(x)=D(K)$, or $\omega_f(x)=S(K)$  , then the periods of cycles in $\mathcal{C}(x)$ are bounded from above and $K$ is minimal in $\mathcal{C}(x)$ with respect to inclusion.		
    
{In order to get a better picture of a circumferential set, one may recall the well-known \emph{Denjoy map} (see~\cite[Example~14.9]{Devaney}).
Roughly speaking, we start with an irrational rotation $R$ on the circle $\mathbb{S}^1$ and take any point $z\in \mathbb{S}^1$. 
We cut out each point of the full orbit of $z$ and replace $R^n(z)$ with interval $I_n$ for all $n\in\mathbb{Z}$, keeping \[\sum_{n=-\infty}^\infty I_n<\infty.\]
Finally, we extend $R$ to the union of $I_n$-s by taking any orientation-preserving diffeomorphism $h_n$ taking $I_n$ to $I_{n+1}$ for every $n\in\mathbb{Z},$
obtaining map $\hat{R}$ on the new circle $\hat{\mathbb{S}^1}$.
The obtained system obviously has no periodic points and all the points from the interiors of $I_n$-s are wandering.
Furthermore, it is easy to see that $\omega_{\hat{R}}(x)=M$, for every $x\in \hat{\mathbb{S}^1}$, where $M$ is a Cantor set $M=\hat{\mathbb{S}^1}\setminus\bigcup_{n\in\mathbb{Z}} \Int I_n$. 
Therefore, $M=S(\hat{\mathbb{S}^1},\hat{R})$ is a circumferential set, almost conjugated to irrational rotation $R$ via projection $\pi\colon \hat{\mathbb{S}^1}\to \mathbb{S}^1$ which collapses every $I_n$ to the respective $R^n(z)$.
In fact, any surjective graph map without periodic points (and hence any circumferential set) is obtained as the almost 1-1 extension of the irrational rotation on the circle. {Similarly to Denjoy map, circumferential sets can be viewed as blow ups of some backward invariant countable subset of the circle into intervals, where the maps $I_n
\to I_{n+1}$ are not necessarily invertible, and with the possible exception of finitely many points from that set which may come from some other graphs than intervals, thus allowing more complicated dynamics on them (see~\cite{Shao} for a more detailed insight).}
Therefore, a topological graph admits a map with circumferential set if and only if it contains a simple closed curve and, moreover, there can be only as many circumferential sets as there are disjoint simple closed curves in that graph~\cite{B1}.
}
    
If $\omega_f(x)$ is solenoidal set then there exists a sequence $\left(K_n\right)_{n\geq 1}$ of cycles of graphs in $\mathcal{C}(x)$ with strictly increasing periods $\left(p_n\right)_{n\geq 1}$ such that, for all $n\geq 1$, $K_{n+1}\subset K_n$ and $\omega_f(x)\subset\bigcap_{n\geq 1}K_n$.
	Furthermore, for all $n\geq 1$, $p_{n+1}$ is a multiple of $p_n$ and every connected component of $K_n$ contains the same number (equal to $p_{n+1}/p_n\geq 2$) of components of $K_{n+1}$. 

\subsection{$\alpha$-limit sets in graph maps}

	In~\cite{Forys} the authors presented almost complete characterization of $\alpha$-limit sets of a graph map.
	In particular, in case of a zero-entropy graph map the situation is clear: the family of $\alpha$-limit sets of backward branches coincides with the family of minimal sets.
	On the other hand, when a graph map has positive topological entropy and, hence, there is a basic set, then the $\alpha$-limit sets can also be some infinite $\omega$-limit sets contained in maximal basic sets, alongside with circumferential sets, periodic orbits and minimal subsets of solenoids as in the first case.

{In what follows, we will also need the following classical result by Blokh \cite{B1}.

\begin{theorem}\label{thm:blokh-almost}
Let $f \colon G \to G$ be a graph map and $K\subset G$ a finite union of subgraphs
such that $f (K) \subset K$. Suppose that $M= E(K, f )$ is infinite. Then $M$ is a perfect set, $f |_M$ is
transitive (i.e., $M$ is an orbit enclosing $\omega$-limit set) and, for every $z\in G$, if $\omega_f (z) \supset M$ then $\omega_f (z) = M$
(hence, $M$ is a maximal $\omega$-limit set). Moreover, there exists a transitive map $g \colon Y \to Y$ , where
$Y$ is a finite union of graphs, and a semi-conjugacy $\phi : K \to Y$ between $f|_K$ and $g$ which almost
conjugates $f |_M$ and $g$.
\end{theorem}

\begin{remark}\label{rem:cir-rot}
Let $M=E(K,f)$ be a circumferential set. Fix a component $C$ of $K$  with period $r$. Then
 an almost conjugacy provided by Theorem~\ref{thm:blokh-almost} can in practice be regarded as
 $\phi\colon (C,f^r)\to (\mathbb{S}^1,R)$, where $R$ is an irrational circle rotation (see \cite{Shao} for more comments and related results).
\end{remark}
}

\subsection{Hyperspaces}

	A \emph{hyperspace} of a metric space $(X,\dist)$ is a specified family of nonempty closed subsets of $X$. 
	The hyperspaces which will be of our interest are $2^X$, the hyperspace of all nonempty compact subsets of $X$ and $C(X)$ which is the hyperspace of all continua, i.e. the connected elements of $2^X$.
	For a point $x\in X$ and a nonempty subset $A\subset X$ we define the \emph{distance from the point $x$ to the set $A$} as $\dist(x,A)=\inf\{\dist(x,y)\colon y\in A\}$
	and for each nonempty set $A\subset X$ and each $\epsilon>0$, we define the \emph{$\epsilon$-neighborhood of the set $A$} as $N(A,\epsilon)=\{x\in X\colon \dist(x,A)<\epsilon\}$.
	We can now endow $2^X$ with a function $\dist_H\colon 2^X\times 2^X\to\mathbb{R}$, defined for each pair $A,B\in 2^X$ as:
	\[\dist_H(A,B)=\inf\{\epsilon\geq 0\colon {A\subset N(B,\epsilon)}\text{ and }B\subset N(A,\epsilon)\}.\] 
	It is well known that $\dist_H$ is a metric, called \emph{Hausdorff metric}, and that $(2^X,\dist_H)$ is a compact metric space (e.g. see \cite{Nadler}). 
	It is also well know that $C(X)$ is a compact subset of $2^X$. 
	The topology generated by $\dist_H$ does not depend on the metric $\dist$ but only on the topology that $\dist$ generates on $X$.

	If $f\colon X\to X$ is a continuous map, then it naturally extends to $\bar{f}\colon 2^X\to 2^X$ via $\bar{f}(A)=f(A)$.
	It is well known that $\bar{f}$ defined this way is a continuous map on $2^X.$
	Obviously, $C(X)$ is $\bar{f}$-invariant subset of $2^X$ and we may denote $\tilde{f}=\bar{f}\vert_{C(X)}$.
	This way we obtain two \emph{induced systems} $(2^X,\bar{f})$ and $(C(X),\tilde{f})$ generated by $(X,f)$.
	Observe that $(X,f)$ may be considered a subsystem (that is, a closed invariant subset) of $(C(X),\tilde{f})$ by identifying $x\in X$ with $\{x\}\in C(X)$,
	and thus also a subsystem of $(2^X,\bar f)$. 
	{Note that $\mathcal{O}_f(A)\neq \mathcal{O}_{\tilde f}(A)$, because the first orbit is a subset of $X$, while the second one is a set of subsets of $X$.}

\subsection{Recurrence in the induced system of continua of topological graphs}

 Let $M=E(K,f)$ be a circumferential set. Fix a component $C$ of $K$  with period $r$ and let
 $\phi\colon (C,f^r)\to (\mathbb{S}^1,R)$ be an almost conjugacy with an irrational rotation $R$ {(recall~\cite[Theorem 3]{B1})}.
 Denote by $\mathcal{R}_C\subset C(G)$ the set consisting of $C$ and subcontinua $A$ of $C$ 
 which satisfy the following two conditions:
\begin{enumerate}
 	\item $A$ has exactly two boundary points $x,y$ in the relative topology of $C$ and $x,y\in M$
 	\item $\phi^{-1}(\phi(x))\subset A$ and $\vert\phi^{-1}(\phi(y))\cap A\vert=1$ or vice-versa.
\end{enumerate}
Denote by $\mathcal R$ the union of all $\mathcal{R}_C$ over all circumferential sets $M$ and related components $C$.
	The authors in~\cite{Jelic} recently proved the following.
\begin{theorem}\label{thm:rec-cont}
Let $G$ be a topological graph and let $f$ be a graph map on $G$.
	Then 
\begin{equation}
\Rec(\tilde{f})=\Per(\tilde{f})\cup \{\{x\}\colon x\in \Rec(f)\}\cup \mathcal{R}.\label{eq:recAdescription}
\end{equation}
\end{theorem}


\section{The proof of Theorem~\ref{thm:main}}\label{sec:proof}	
We will be using the following classical result (see \cite{Furst}).
It will help us to synchronize dynamics of some points.
\begin{theorem}[Auslander-Ellis]\label{thm:auslander}
	In a dynamical system on a compact metric space, every
point is proximal to a uniformly recurrent point in its orbit closure.	
\end{theorem}
	
An immediate consequence of Auslander-Ellis theorem is the following fact, which will be of great utility in our arguments.
\begin{corollary}\label{cor:proximal}
Let $(X,f)$ be a dynamical system on a compact metric space $X$.
	For any $x\in X$ there exists $y\in\Rec(f)$ and an increasing sequence $\{n_k\}_{k\geq 0}$ of nonnegative integers such that \[(\forall\epsilon>0)(\exists k_0\geq 0)(\forall k\geq k_0) \dist(f^{n_k}(x),y)<\epsilon \text{ and } \dist(f^{n_k}(y),y)<\epsilon.\]
\end{corollary}

Although this is not explicitly stated in the assumptions, for each result contained
in this section we assume that we are given a dynamical system $(G,f)$ on some arbitrary
topological graph $G$. Let us also recall that we refer to Remark~\ref{rem:ambient}, which provides
standing assumption on the metric on $G$ for the whole section.

The following observation is straightforward.
\begin{lemma}\label{lem:iterates}
	Let $A$ be a continuum. 
If there are $k\geq 0$ and $n\geq 1$ such that $f^k(A)$ is asymptotically periodic under ${\tilde{f}}^n$ then $A$ is asymptotically periodic under $\tilde{f}$.
\end{lemma}

{ We will also need the following lemma. Details of the proof are simple and left to the reader (cf. \cite[Lemma~3.2]{Jelic}).}

\begin{lemma}\label{lem:circ}
	Let $A\in\C(G)$ and let $M=E(K,f)$ be a circumferential set such that $\omega_f(x)=M$ for all $x\in A$.
	Then there is some $k\geq 0$ such that $f^k(A)\subset K$.
\end{lemma}

\begin{lemma}\label{lem:per_pts}
Let $A\in\C(G)$ be not wandering for $f$, that is, let there be some $0\leq r<t$ such that $f^r(A)\cap f^t(A)\neq\emptyset$.
	Then at least one of the following properties holds.
	\begin{enumerate}[(i)]
\item\label{lem:per_pts:c1} There exist some $m\geq 0$ and some periodic point $y$ such that $y\in f^m(A)$.
\item\label{lem:per_pts:c2} There exists some $x\in A$ such that $\omega_f(x)$ is circumferential.
\item\label{lem:per_pts:c3} There exists some finite set $B$, which is a union of periodic orbits, such that $\omega_f(x)\subset B$ for all $x\in A$. 	
	\end{enumerate}
\end{lemma}
\begin{proof}
Suppose that \eqref{lem:per_pts:c1} and \eqref{lem:per_pts:c2} do not hold and let us prove \eqref{lem:per_pts:c3}.
	We can assume that $r,t$ are the minimal distinct nonnegative integers such that $f^r(A)\cap f^t(A)\neq\emptyset$.
	It follows that, for each $i=0,1,...,t-r-1$, the set $\bigcup_{n=0}^\infty f^{i+r+n(t-r)}(A)$ is connected.
	Together with $A,f(A),...,f^{r-1}(A)$, we conclude that $K=\bigcup_{n=0}^\infty f^n(A)$ has at most $r+t-r=t$ connected components.
	The set $\overline{K}=\overline{\bigcup_{n=0}^\infty f^n(A)}$ is therefore an invariant closed set with finite number of components.
	Obviously, $\omega_f(x)\subset \overline{K}$ for all $x\in A$.
	Since $(i)$ does not hold, $\overline{K}\cap\Per(f)\subset\partial K$ and hence it is finite.
	
	First suppose that for some $x\in A$, the set $\omega_f(x)$ is contained in a maximal solenoid.
	Recall that there is a minimal set $\omega\subset\omega_f(x)$ and  $\omega=\omega_f(x)\cap\overline{\Per(f)}$.
	It follows that $\overline{K}\cap\Per(f)$ is infinite, a contradiction.
	
	Let us now suppose that $\omega_f(x)$ is an infinite set contained in a basic set $\omega$ for some $x\in A$.
	It is well known (see~\cite[Lemma~6]{malek}) that $\omega\subset\overline{\Per(f)}$ and we again conclude that $K\cap\Per(f)$ is infinite, a contradiction.
	
	Therefore, since the case~\eqref{lem:per_pts:c2} is excluded, for each $x\in A$, $\omega_f(x)$ is a periodic orbit and $B=\bigcup_{x\in A}\omega_f(x)$ is obviously a finite set.
\end{proof}

	In Lemma~\ref{lem:degenerate_d} we describe the case when Corollary~\ref{cor:proximal} provides us with some degenerate continuum $D=\{d\}\in\Rec(\tilde{f})$.
	Lemma~\ref{lem:d-not-contained} proves the main theorem when such $\{d\}$ is disjoint with $\bigcup_{n=0}^\infty f^n(A)$ .
	
\begin{lemma}\label{lem:degenerate_d}
Let $A\in\C(G)$ and suppose that for $A$ there exists some degenerate $D=\{d\}\in\Rec(\tilde{f})$, satisfying Corollary~\ref{cor:proximal}.
	Then $A$ either satisfies property~\eqref{wandering} or~\eqref{circ} of Theorem~\ref{thm:main} or $d$ is a periodic point and for all $x\in A$, $\omega_f(x)=\mathcal{O}(d).$ 

\end{lemma}
\begin{proof}
	Suppose $A$ does not satisfy the property~\eqref{wandering} nor~\eqref{circ} of Theorem~\ref{thm:main} and let us prove the assertion.
	
	Since, for all $x\in A$, $d\in\omega_f(x)$, we conclude that $d$ is not an element of a circumferential set and that $\omega_f(x)$ is not a circumferential set for any $x\in A$. 
	Indeed, it is not hard to see that if $d$ was an element of a circumferential set $M$ then we would have $M=\omega_f(x)$ for all $x\in A$ and by Lemma~\ref{lem:circ}, the condition~\eqref{circ} of Theorem~\ref{thm:main} would be satisfied which is in contradiction with our assumption.	
	
	Since $A$ is non-wandering, we can apply Lemma~\ref{lem:per_pts}.
	Continuum $A$ does not satisfy the property \eqref{lem:per_pts:c2} of the Lemma and in both other cases of the Lemma, there is a point $x\in A$ and a periodic orbit $P=\{y,f(y),...,f^{p-1}(y)\}$ such that $\lim_n\dist(f^n(x),P)=0$.
	It easily follows that $d\in P$.

	We have yet to prove that for all $x\in A$, $\omega_f(x)=\mathcal{O}(d)$.
	First note that, since $A$ is non-wandering and since condition~\eqref{lem:per_pts:c2} of Lemma~\ref{lem:per_pts} does not hold, we are left with conditions~\eqref{lem:per_pts:c1} and~\eqref{lem:per_pts:c3}.
	If condition~\eqref{lem:per_pts:c3} is satisfied then, since $d\in \omega_f(x)$, we have that $\omega_f(x)=\mathcal{O}(d)$ for all $x\in A$.
	Suppose on the contrary, i.e. that there is some $u\in A$ such that $\omega_f(u)$ is an infinite subset of a basic set and $d\in\omega_f(u)$.
	Then the condition~\eqref{lem:per_pts:c1} of Lemma~\ref{lem:per_pts} holds, i.e. there exist some $m\geq 0$ and some periodic point $z$ such that $z\in f^m(A)$.
	Clearly, $z\in\mathcal{O}(d)$, otherwise $\dist_H(f^n(A),\{d\})\geq \dist(\mathcal{O}(z),d)>0$ for all $n\geq m$.
	Let $p$ be the period of $d$.
	Note that, since $z\in f^{m+np}(A)$ for all $n\geq 0$ and since $\omega_{f^p}(f^m(u))$ is an infinite subset of a basic set, which is contained in $\overline{\Per(f)}$, then there is some \[w\in \left(\Per(f)\setminus\mathcal{O}(d)\right)\cap f^{m+rp}(A)\] for some $r\geq 0$.
	But this contradicts the earlier observation that $d\in\omega_f(x)$ for all $x\in A$.	

Indeed, for all $x\in A$, $\omega_f(x)=\mathcal{O}(d)$.
\end{proof}

\begin{lemma}\label{lem:self-covering}
	Let $A\in\C(G)$ and suppose there is some continuum $D\in\omega_{\tilde{f}}(A)$ which is a fixed point of $\tilde{f}$.
	If there are some $m,k\geq 0$ and some continuum $B\subset f^m(A)$ such that $B\subset f^k(B)$ then $B\subset D$.
\end{lemma}

\begin{proof}
	Suppose on the contrary, that $B\setminus D\neq\emptyset$.
	Take some point $y\in B\setminus D$ and set $\epsilon=\dist(y,D)>0$.
	Notice that $\dist_H(C,D)<\epsilon$ implies $y\notin C$.
	Since $\tilde{f}(D)=D$, by using the continuity of $\tilde{f}$ $k-1$ times, we find some $\delta>0$ such that for each $C\in\C(G)$, $\dist_H(C,D)<\delta$ implies $\dist_H(f^t(C),D)<\epsilon$ for $t=0,1,...,k-1$.
	Since $D\in\omega_{\tilde{f}}(A)$, there is some $r>m$ such that $\dist_H(f^r(A),D)<\delta$.
	Notice that for all $n\geq 0$, $B\subset f^{m+nk}(A)$. 
	Therefore we find some $0\leq j<k$ such that $y\in B\subset f^{r+j}(A)$, while $\dist_H(f^{r+j}(A),D)<\epsilon$, a contradiction.
\end{proof}

{
Lemma~\ref{lem:degenerate_d} {covers the case} 
when there is a periodic point $d\in G$ such that, for all $x\in A$, $\omega_f(x)=\{d\}$. 
{Intuitively it seems that this implies $\omega_{\tilde{f}}(A)=\mathcal{O}_{\tilde{f}}(\{d\})$, however unfortunately it is not always the case.} In the following example we provide a graph map $f\colon G\to G$ with nondegenerate fixed subgraph $A$ and a fixed point $d\in A$, such that $\omega_f(x)=\{d\}$ for all $x\in A$.
\begin{example}
Let $I=[0,1]$ and let $g\colon I\to I$, $g(x)=x^2$. 
Note that $0$ and $1$ are fixed points and that $\lim_n g^n(x)=0$ for $x\neq 1$.
$M=\{0,1\}$ is a closed invariant set so we can define a factor system $(G,f)$ where $G=I/_M$ is homeomorphic to a unit circle and $f=g/_M$ has a single fixed point $d$, with $\{d\}=\pi_M(M)$, such that $\bigcup_{x\in G}\omega_f(x)=\{d\}$.
But $f$ is onto so $G\in\Fix\left(\tilde{f}\right)$ and hence $\omega_{\tilde{f}}(G)=\{G\}$.
In fact, for any nondegenerate $A\in\C(G)$ such that $d\in A$, we have $\omega_f(x)=\{d\}$ for all $x\in A$ and $\omega_{\tilde{f}}(A)=\{G\}$.
\end{example}
}

\begin{lemma}\label{lem:d-not-contained}
Let $A\in\C(G)$.
	Suppose that for $A$ there exists some degenerate $D=\{d\}\in\Rec(\tilde{f})$, satisfying Corollary~\ref{cor:proximal}.
	Furthermore, suppose that $d\notin\bigcup_{n=0}^\infty f^n(A)$.
	Then $A$ either satisfies property~\eqref{wandering} or~\eqref{circ} of Theorem~\ref{thm:main} or $\omega_{\tilde{f}}(A)=\mathcal{O}(\{d\})$.
\end{lemma}

\begin{proof}
By Lemma~\ref{lem:degenerate_d} we know that $A$ either satisfies property~\eqref{wandering} or~\eqref{circ} of Theorem~\ref{thm:main} or $d$ is periodic and 
for all $x\in A$, $\omega_f(x)=\mathcal{O}(d).$	  
Suppose the latter case and let us show that $\omega_{\tilde{f}}(A)=\mathcal{O}(\{d\})$.
 Since $d$ is periodic with some minimal period $p$, there is some $\epsilon^\prime$ such that $\dist(d,f^n(d))<\epsilon^\prime$ implies $d=f^n(d)$, that is $p$ divides $n$.
  For increasing sequence $(n_k)$ provided by Corollary~\ref{cor:proximal} we have that \[(\exists k_0\geq 0)(\forall k\geq k_0) \dist_H(f^{n_k}(A),\{d\})<\epsilon^\prime \text{ and } \dist(f^{n_k}(d),d)<\epsilon^\prime,\]
 meaning that all but finitely many elements of $(n_k)$ are multiples of $p$.
	This, alongside with Lemma~\ref{lem:iterates} allows us to replace $f$ with $f^p$, i.e. to assume that $d$ is a fixed point. 
	
	We claim that for each $x\in G$, the set $\{n\geq 0\colon x\in f^n(A)\}$ is finite.
	Suppose on the contrary, i.e. that there is some $y\in G$ not satisfying this property.
	Note that obviously $y\neq d$.
	For each $\epsilon>0$, for all but finitely many $k$ we have $f^{n_k}(A)\subset B(d,\epsilon)$.
	If $\epsilon$ is small enough and $f^{n}(A)\subset B(d,\epsilon)$, $d\notin f^{n}(A)$   for some $n\geq 0$ then $f^{n}(A)$ is contained in some one-sided neighborhood of $d$.
	We can pick the sides of $d$ in the way that for each side $T$, either each element of $T$ contains $f^{n}(A)$ for infinitely many $n$, or each element of $T$ does not contain $f^n(A)$ for any $n$.
	Let $T_1,T_2,...,T_r$ be the sides of $d$ with the property described above and for each $i$, let $\zeta_i$ be the supremum over lengths of elements of $T_i$.
	Finally, put $\zeta=\min\{\dist(d,y),\zeta_1,\zeta_2,...,\zeta_r\}$.
After renumeration, suppose $K_1\in T_1,K_2\in T_2...,K_v\in T_v$, for some $v\leq r$, are all the one-sided neighborhoods of length $\zeta$ containing $f^n(A)$ for infinitely many $n$.
	By continuity of $f$, since $d$ is a fixed point, there exists some $0<\epsilon<\zeta$ such that $f(B(d,\epsilon))\subset B(d,\zeta)$.
	For each $i=1,2,...,v$ we find some $k_i\geq 0$ such that $f^{k_i}(A)\subset B(d,\epsilon)\cap K_i$.
	Let $\delta=\min\{\dist(d,f^{k_i}(A)),\ i=1,2,...,v\}>0$.
Pick some $x\in A$. While now there are $p$ fixed points, still
	by Lemma~\ref{lem:degenerate_d} there is some $n_0\geq 0$ such that $n\geq n_0$ implies $f^n(x)\in B(d,\delta)$, since all points of $A$ are attracted by $d$.
There is some $n_t>\max\{n_0,k_1,k_2,...,k_v\}$ such that $f^{n_t}(A)\subset B(d,\delta)$.
By assumption, $f^{n_t}(A)$ contains some point in the negative orbit of $y$.
This means there is some minimal $m$ such that $f^{n_t+m+1}(A)\setminus B(d,\zeta)\neq\emptyset$.
	This means $f^{n_t+m}(A)\subset B(d,\zeta)$ and $f^{n_t+m}(A)\setminus B(d,\epsilon)\neq\emptyset$.
	Since $d\notin f^{n_t+m}(A)$, there is some $i\in\{1,2,...,v\}$ such that $f^{n_t+m}(A)\subset K_i$.
	Moreover, 
	$ f^{n_t+m}(x)\in f^{n_t+m}(A)\cap B(d,\delta)$ and $f^{n_t+m}(A)\setminus B(d,\epsilon)\neq\emptyset$, meaning $f^{k_i}(A)\subset f^{n_t+m}(A)$ and $k_i<n_t+m$.
	This is in contradiction with Lemma~\ref{lem:self-covering}, {because $\{d\}\neq f^{k_i}(A)$}.
Indeed, for each $x\in G$, the set $\{n\geq 0\colon x\in f^n(A)\}$ is finite, that is, the claim holds.

Now to finish the proof, take any $\kappa>0$ and some $z\in A$.
	We have that $f^n(z)\in B(d,\kappa)$ for all but finitely many $n$. 
	If $f^n(A)\setminus B(d,\kappa)\neq\emptyset$ for infinitely many $n$ then there would be some point contained in $\partial B(d,\kappa)\cap f^n(A)$ for infinitely many $n$, contradicting this way the claim from the previous paragraph.
	Therefore, for all $\kappa> 0$ we have $f^n(A)\subset B(d,\kappa)$ for all but finitely many $n$. The proof is completed.	
	\end{proof}

In the series of auxiliary results presented below we will mostly assume that, for given $A\in\C(G)$, there is some $D\in\C(G)$, which is a fixed point of $(\C(G),\tilde{f})$ and a unique periodic element of $\omega_{\tilde{f^n}}(A)$ for every $n\geq 1$.
Although this may seem like a serious restriction, Lemma~\ref{lem:periodic_supset} will show that $D\in\Rec(\tilde{f})$ provided Corollary~\ref{cor:proximal}, in the most of the cases satisfies this assumption. Namely, it holds whenever $D$ provided Corollary~\ref{cor:proximal} is either nondegenerate or it is a periodic singleton contained in $f^m(A)$ for some $m\geq 0$.

	The following result, proved in~\cite{Jelic}, concerns circumferential sets and will be of help in the proof of Lemma~\ref{lem:circumferential}.
	
\begin{lemma}\label{lem:contains_circumferential}
If a nondegenerate continuum $A$ is recurrent and there is $x\in A$ such that $\omega_f(x)\neq S(K,f),$ where $S(K,f)$ is a circumferential set 
then $C\subset A$ for each component $C$ of $K$ intersecting $A.$
\end{lemma}	
	
\begin{lemma}\label{lem:circumferential}
	Let $A\in\C(G)$ and let $D$ be a fixed point of $(\C(G),\tilde{f})$. 
	Suppose that there is some circumferential set $M=E(K,f)$ in $D$.
	Then for all $y\in K$, $\lim_n\dist(y,f^n(A))=0$.
	\end{lemma}	
\begin{proof}
	If $\omega_f(z)=M$ for all $z\in A$ then, for some $k\geq 0$, $f^k(A)\subset K$ by Lemma~\ref{lem:circ}.
 Then necessarily $D=K$
	and $\mathbb{S}^1\in\omega_{\tilde R}(\phi(f^k(A)))$, where $\phi$ is an almost conjugacy {(see Remark~\ref{rem:cir-rot})}.
	Since $R$ is isometry, $\phi(f^k(A))=\mathbb{S}^1$.
This means that for each $x\in K$, $f^k(A)$ contains all but at most finitely many elements in the backward orbit of $\phi^{-1}(\phi(x))$ and therefore $(f^n(A))\to K$ as $n\to\infty$.

	Now suppose that for some $z\in A$, $\omega_f(z)\neq M$.
	Let $p$ be the period of $K$. If we prove that for each component $C$ of $K$ and each $y\in C$, $\lim_n(y,f^{np}(A))=0$ then we are done.
	Since $D\in\omega_{\tilde{f}^p}(A)$, we may assume that $K$ is connected.
	Also note that there is some $n_0\geq 0$ such that $f^n(A)\cap K\neq\emptyset$ for all $n\geq n_0$ and that $\Per(f)\cap K=\emptyset$.	
	
	We claim that there is an arc $L$ and some $j\geq 0,m\geq 1$, such that $L\subset f^j(A)$, $L\subset f^m(L)$ { and $K\cap L\neq\emptyset$}.

{First, let us show that} if $\left(\bigcup_{n=0}^\infty f^n(A)\right)\cap \Per(f)=\emptyset$ then $\omega_f(z)$ is a circumferential set or a periodic orbit.
	Suppose on the contrary, that $\omega_f(z)$ is a basic or a solenoidal $\omega$-limit set. 
    {Note that $\omega_f(z)\cap K=\emptyset$ since $f\vert_K$ has the unique minimal set $M$  (recall~\cite[Corollary 2.2]{Shao}) and $M\cap\omega_f(z)=\emptyset$ {because maximal $\omega$-limit sets of different type never intersect}.}
	Since $f^n(A)\cap K\neq\emptyset$ for all $n\geq n_0$, $\omega_f(z)\cap K=\emptyset$ and since $\omega_f(z)$ is an infinite subset of $\overline{\Per(f)}$ in the case of basic set or there is again an infinite $Q_{min}\subset \omega_f(z)$ and $Q_{min}\subset\overline{Per(f)}$ in the case of solenoid, we easily see that there is some $n_1\geq 0$ such that $f^{n_1}(A)\cap\Per(f)\neq\emptyset$,
	{which leads to contradiction} with $\left(\bigcup_{n=0}^\infty f^n(A)\right)\cap \Per(f)=\emptyset$.

	
{We will divide the rest of the proof in two complementary cases.
We begin with the first of those two cases, namely that }either we have $f^{n_1}(A)\cap\Per(f)\neq\emptyset$ for some $n_1\geq 0$ or that 
$\omega_f(z)$ is a circumferential set. {In the former case let $K^\prime$ be the orbit of some periodic point from $f^{n_1}(A)$, while in the latter let $K^\prime$ be the minimal cycle of graphs containing $\omega_f(z)$}. In both cases, we can choose a (possibly degenerate) cycle of graphs $K^\prime$ such that $K^\prime\cap K=\emptyset$ and find some $n_1\geq 0$ such that $K^\prime\cap f^n(A)\neq\emptyset$ for all $n\geq n_1$.
	Now we consider a family of arcs $\{L_i\colon i\geq {n_1}\}$, contiguous to $K$ and $K^\prime$, such that $L_i\subset f^i(A)$, $L_{i+1}\subset f(L_i)$.
	There are finitely many such arcs so we can indeed find an arc $L$ and some $j\geq 0,m\geq 1$, such that $L\subset f^j(A)$ and $L\subset f^m(L)$.
	
	Now {let us consider the second case, that is, let us assume} that \[\left(\bigcup_{n=0}^\infty f^n(A)\right)\cap \Per(f)=\emptyset\] and that $\omega_f(z)$ is not circumferential.
	By what is proved above, $\omega_f(z)$ is a periodic orbit of period $r\geq 1$ so $\omega_{f^r}(z)=\{w\}$ for some point $w$ which is fixed under $f^r$.
	There is a sequence $\left(n_k\right)_{k\geq 1}$ of positive integers such that $\left(f^{n_k}(z)\right)_{k\geq 1}$  is a monotone sequence in an edge of $G$.
	Now we define a sequence $\left(L_{n_k}\right)_{k\geq 1}$ of arcs such that each $L_{n_k}$ intersects $K$ only at its endpoint, its other endpoint is $f^{n_k}(z)\in f^{n_k}(A)$ and $L_{n_{k+1}}\subset f^{n_{k+1}-n_k}(L_{n_k})$. 
	 Since $w\notin\left(\bigcup_{n=0}^\infty f^n(A)\right)$, by similar arguments as before we find an arc $L$ and some $j\geq 0,m\geq 1$, such that $L\subset f^j(A)$ and $L\subset f^m(L)$. {Indeed, the claim holds.}
	
	Now, $J=\overline{\bigcup_{n=0}^\infty f^{nm}(L)}$ is a continuum fixed under $f^m$, {$K\cap L\neq \emptyset$} and there is a point in $J$ with $\omega$-limit set different from $M$ hence, by Lemma~\ref{lem:contains_circumferential}, $K\subset J$.
	 By the construction, for all $y\in K$, $\lim_n\dist(y,f^{k+nm}(A))=0$ and, since $K$ is {a fixed  point of $\tilde f$} and $f$ is uniformly continuous, $\lim_n\dist(y,f^{n}(A))=0$, for all $y\in K$.
\end{proof}

The following is a trivial consequence of Lemma~\ref{lem:circumferential} and continuity of $f$.
\begin{corollary}\label{cor:alpha_circumferential}
Let $A\in\C(G)$ and let $D$ be a fixed point of $(\C(G),\tilde{f})$. 
	Suppose that there is some $x\in D$ with backward branch $\{x_j\}_{j\leq 0}$ in $D$ such that $\alpha(\{x_j\}_{j\leq 0})\subset D$ is a circumferential set.
	Then $\lim_n\dist(x,f^n(A))=0$.	
\end{corollary}

The following Lemma~\ref{lem:solenoid-components} is again a result from~\cite{Jelic} and will be used in the proof of Lemma~\ref{lem:alpha_solenoid}.
	
\begin{lemma}\label{lem:solenoid-components}
	If a nondegenerate continuum $A$ is recurrent and there is $x\in A$ such that $\omega_f(x)$ is contained in a maximal solenoid $\omega,$ then there is a cycle of graphs $K$ containing $\omega$ with the property that for each $y\in\omega\cap A,$ $A$ contains a component of $K$ to which $y$ belongs.
\end{lemma}

\begin{lemma}\label{lem:alpha_solenoid}
Let $A\in\C(G)$ and
let $D$ be a nondegenerate fixed point of $(\C(G),\tilde{f})$. 
Let $w\in D$ be a point with backward branch $\{w_j\}_{j\leq 0}$ in $D$ such that $\alpha(\{w_j\}_{j\leq 0})\subset D$ is a minimal solenoid $Q_{min}$.
Then $\lim_n\dist(w,f^n(A))=0$.	
\end{lemma}	

\begin{proof}
Let us denote by $\{K_i\}_{i\geq 1}$ the nested sequence of cycles of graphs such that $Q=\bigcap_{i}K_i$, $Q_{min}=Q\cap\overline{\Per f}$ and denote $\omega=Q_{max}=Q\cap\omega(f)$.
	The idea of the proof is to show that there is a cycle of graphs $K\in\{K_i\}_{i\geq 1}$ such that for all $x\in K$, $\lim_n\dist(x,f^n(A))=0$.
	Since $f(K)=K$, it is enough to show that for some $p\geq 0$ and for all $x\in K$, $\lim_n\dist(x,f^{pn}(A))=0$. 
	For each $i\geq 1$, since $\alpha(\{w_j\}_{j\leq 0})=Q_{min}\subset K_i$, $Q_{min}$ is infinite and $K_i$ is invariant, we have that $w_j\in K_i$ for all {$j\leq 0$}. 
	Specially, $w_j\in K$ for all {$j\leq 0$} 
	and hence $\lim_n\dist(w,f^n(A))=0$.

Denote by $m$ the number of edges in $G$ and take some $r\geq 1$ such that $K_r$ has at least $2m+1$ components.
Then there is an edge $E$ in $G$ such that at least $3$ components of $K_r$ intersect it and, since those components are pairwise disjoint, there is a component of $K_r$ which is a subarc of $\Int E$.
	Denote by $K_r^1,K_r^2,...,K_r^{p_r}$ the components of $K_r$ in a way that $K_r^1\subset \Int E$, $f(K_r^i)=K_r^{i+1}$ for $1\leq i<p_r$ and $f(K_r^{p_r})=K_r^{1}$.
If, for each $i,\ 1\leq i\leq p_r$, we show that there is some $r_i\geq r$ such that for all $y\in K_r^i\cap K_{r_i}$,  $\lim_n\dist(y,f^{np_r}(A))=0$ then for $K$ we can take any $K\in\{K_i\}_{i\geq 1}$, $K\subset\cap_{i=1}^{p_r} K_{r_i}$ and the lemma will be proved.
	Therefore, take some $i,\ 1\leq i\leq p_r$ and let us find such $K_{r_i}$.
	Note that $f^{p_r-i+1}(K^i_r)=K_r^1$. If we find some $K_{r_i}$ such that for each $y\in K_{r_i}\cap K_r^1$, $\lim_n\dist(f^{p_r-i+1+np_r}(A),y)=0$ then, since $f^{i-1}(K_{r_i}\cap K_r^1)=K_{r_i}\cap K_r^i$, for each $y\in K_{r_i}\cap K_r^i$, we will have $\lim_n\dist(f^{np_r}(A),y)=0$ and will be done.
	
	Take some $r_i^\prime\geq p_r$ such that $K_{r_i^\prime}\cap K_r^1$ has at least three components and note that $K_{r_i^\prime}\cap K_r^1$ is fixed under $f^{p_r}$.
	Since $D\in\omega_{\tilde{f}^{p_r}}(f^{p_r-i+1}(A))$ and since $K_r^1$ is a subarc of interior of an edge of $D$, there is some $t\geq 0$ such that $f^{p_r-i+1+tp_r}(A)$ contains all but at most one component of $K_{r_i^\prime}\cap K_r^1$.
	If it contains all of its components then we are done since in that case $K_{r_i^\prime}\cap K_r^1\subset f^{p_r-i+1+np_r}(A)$ for all $n\geq t$ and we can set $r_i=r_i^\prime$.
	Now suppose that $f^{p_r-i+1+tp_r}(A)$ contains all but one component of $K_{r_i^\prime}\cap K_r^1$.
	After increasing $t$ if necessary, we can assume that a component of $K_{r_i^\prime}\cap K_r^1$
which is not contained in $f^{p_r-i+1+tp_r}(A)$, in the ordering of $E$, lies between some other two components of the same set which are contained in $f^{p_r-i+1+tp_r}(A)$. {This is possible since $f^{p_r}$ fixes $K_{r_i^\prime}\cap K_r^1$ and permutes its components (see~\cite[Lemma 10]{Snoha}).}

	Let us denote by $C_1,C_2,...,C_s$ the components of $K_{r_i^\prime}\cap K_r^1$ in a way that $C_1,C_2,C_3$ are subsequent in the ordering of $E$ and that $C_2\not\subset f^{p_r-i+1+tp_r}(A)$. Note that, for all $v,\ 1\leq v\leq s$, $f^{sp_r}(C_v)=C_v$.
	Denote by $q_1,q_2$ the minimal and the maximal point of $\omega$ in the ordering of $C_2$, in a way that $q_1$ lies between $C_1$ and $q_2$ in $E$.
Before continuing, let us show that, after increasing $t$ if necessary, we can suppose that $q_1,q_2\in f^{p_r-i+1+tp_r}(A)$.
Indeed, just like above, we can take some $r_i^{\prime\prime}\geq r_i^\prime$ such that $K_{r_i^{\prime\prime}}\cap C_2$ consists of at least three components and since  $D\in\omega_{\tilde{f}^{sp_r}}(f^{p_r-i+1+tp_r}(A))$, for some $t^\prime$, $f^{p_r-i+1+tp_r+t^\prime sp_r}(A)$ contains all but at most one component of $K_{r_i^{\prime\prime}}\cap C_2$.
Then after increasing $t^\prime$ if necessary, we can assume that $\{q_1,q_2\}\subset f^{p_r-i+1+tp_r+t^\prime sp_r}(A)$.
Indeed, we can suppose that for $t$ defined as above we also have $\{q_1,q_2\}\in f^{p_r-i+1+tp_r}(A)$.

Let $L_1$ be the arc contained in 
 $K_r^1$ intersecting $C_1$ in its endpoint and having $q_1$ as its other endpoint.
Analogously, let $L_2$ also be the arc contained in 
$K_r^1$, but intersecting  $C_3$ in its endpoint and having $q_2$ as its other endpoint, {see Figure~\ref{fig:solenoid}}.  
Note that $L_1\cup L_2\subset f^{p_r-i+1+tp_r}(A)$.
Note that $q_l\neq f^{sp_r}(q_l)\in\omega\cap C_2$ for $l\in\{1, 2\}$ and $C_1,C_3$ are fixed under $f^{sp_r}$. 
This means that $f^{sp_r}(L_1)$ either contains $L_1$ or $L_2$ and $f^{sp_r}(L_2)$ either contains $L_1$ or $L_2$.
If for some $1\leq l\leq 2$ we have $L_l\subset f^{sp_r}(L_l)$ then we are done since in that case $L=\overline{\bigcup_{n=1}^{\infty}f^{nsp_r}(L_l)}$ is a continuum fixed under $f^{sp_r}$ and containing $C_2\cap \omega$ so it contains $K_{r_i}\cap C_2$ for some $r_i\geq r_i^\prime$ by Lemma~\ref{lem:solenoid-components}. This would mean that for all $y\in K_{r_i}\cap K_r^1$, 
$\lim_n\dist(f^{p_r-i+1+tp_r+nsp_r}(A),y)=0$
 and, since $K_{r_i}\cap K_r^1$ is fixed under $f^{p_r}$ and $C_k\subset f^{p_r-i+1+tp_r}(A) $ for $k\neq 2$, $\lim_n\dist(f^{p_r-i+1+np_r}(A),y)=0$ for all $y\in K_{r_i}\cap K_r^1$, proving the lemma.

\begin{figure}[H]
\includegraphics[width=\textwidth]{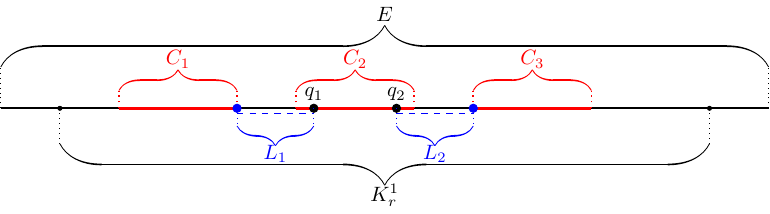}
\caption{Continua $C_i$ from Lemma~\ref{lem:alpha_solenoid}.}\label{fig:solenoid}
\end{figure}

 Therefore, we yet have to consider the remaining case, i.e. that $L_2\subset f^{sp_r}(L_1)$ while $L_1\subset f^{sp_r}(L_2)$.
 This would mean that $L_1\subset f^{2sp_r}(L_1)$ and $L_2\subset f^{2sp_r}(L_2)$.
	Analogously to previous case, we define \[L=\overline{\bigcup_{n=1}^{\infty}f^{2nsp_r}(L_1)} \text{ and } L^\prime=\overline{\bigcup_{n=1}^{\infty}f^{2nsp_r}(L_2)}.\]
	Both $L$ and $L^\prime$ are continua fixed under $f^{2sp_r}$,  $f^{sp_r}(L)=L^\prime$ and therefore $L\cup L^\prime\supset C_2\cap\omega$ {since $f^{sp_r}(L\cup L^\prime)=L\cup L^\prime$.}
	Therefore, we can find some (by taking the smaller of two provided by $L$ and $L^\prime$) cycle $K_{r_i}$ for $r_i\geq r_i^\prime$, such that $L\cup L^\prime\supset K_{r_i}\cap K_r^1$.
	Also, for each $y\in L\cup L^\prime,$ 
 $\lim_n\dist(f^{p_r-i+1+tp_r+2nsp_r}(A),y)=0$ so, because $C_k\subset f^{p_r-i+1+tp_r}(A)$ for $k\neq 2$, 
$\lim_n\dist(f^{p_r-i+1+tp_r+2nsp_r}(A),y)=0$ for each $y\in K_{r_i}\cap K_r^1$.
Since $K_{r_i}\cap K_r^1$ is fixed under $f^{p_r}$, $\lim_n\dist(f^{p_r-i+1+np_r}(A),y)=0$ for all $y\in K_{r_i}\cap K_r^1$ which finishes the proof for the remaining case.
\end{proof}

The following Lemma~\ref{lem:growing_arc} provides an auxiliary result which is used in the proofs of Lemmas~\ref{lem:periodic} and~\ref{lem:alpha_periodic}.

\begin{lemma}\label{lem:growing_arc}
	Let $A\in \C(G)$ and
let $D$ be a nondegenerate fixed point of $(\C(G),\tilde{f})$. 
	Furthermore, let $m$ be the number of edges in  $G$ and let $z\in D$ be a fixed point with a side $T$ such that all the $T$-sided neighborhoods of $z$ are contained in $D$.
	 Suppose that $f^n(A)$ does not contain any $T$-sided neighborhood for any $n\geq 0$.
	Also, suppose that there is some fixed point $y\in A$.
	Then, for each $\epsilon>0$ there is a point $x\in B(z,\epsilon)$, contained in an element of $T$ and some $t\geq 0,\ 1\leq k\leq 2m2^m$, such that 
	{$x\in f^{t+nk}(A)$} for all $n\geq 0$.
\end{lemma}

\begin{proof}
Denote $m^\prime=2m2^m.$
	Take some $0<\delta<\epsilon$ such that there is a side $K\in T$ with length equal to $\delta$.
	Since $L\not\subset f^n(A)$ for all $L\in T$ and all $n\geq 0$, we can pick some ${0<\kappa<\delta}$
	such that \[\bigcup_{n=0}^{2m^\prime}f^n(A)\cap B(z,\kappa)\cap K=\emptyset.\]
	Let \[t=\min\{n\geq 0\colon f^n(A)\cap B(z,\kappa)\cap K\neq\emptyset\}>2m^\prime.\]
	 Note that above defined $t$ exists since $D\in\omega_{\tilde{f}}(A)$ and $K\subset D$. There is $y_0\in A$ such that ${f^t(y_0)\in B(z,\kappa)\cap K}$. 
	Denote $y_i=f^i(y_0)$ for $i\geq 1$. 
	Note that, by 
	{the definition of $t$,} for all $1\leq i\leq t$, $y_i\notin\bigcup_{n=0}^{i-1}f^n(A)$.
 
	There is a family $\{L_i\colon i=0,1,...,t\}$ of arcs, $L_i\subset f^i(A)$, each $L_i$ connecting $y$ and $y_i$, such that $L_{i+1}\subset f(L_i)$.
	For each such arc, there are no more than $2m$ possibilities for the last branching point in the arc, starting from $y$, i.e. branching point in $L_i$ closest to the $y_i$.
	There are also no more than $2^m$ possibilities for the set of edges contained in the arc $L_i$, that is, edges forming the arc, between $y$ and that last branching point.
	By 
	{the pigeonhole} principle and by the choice of points $y_i$, we obtain that for some $0\leq j\leq m^\prime$ and\linebreak $0<k\leq m^\prime$, $L_j\subset f^k(L_j)$.
 Note that $t-j-k\geq 0$ and denote \[C=f^{t-j-k}(L_j)\subset f^{t-k}(A).\]
 {We obtained \begin{align*}
     L_j&\subset f^k(L_j)\\
     f^{t-j-k}(L_j)&\subset f^{t-j-k}(f^k(L_j))=f^{k}(C)
 \end{align*} and
\[f^k(C)=f^k(f^{t-j-k}(L_j))=f^{t-j}(L_j)\ni f^{t-j}(y_j)=f^t(y_0).\]}
 	It follows that $C\subset f^k(C)\ni y_t.$
 	Therefore, \[y_t\in f^{t+nk}(A)\cap B(z,\kappa)\cap K\quad\text{for all}\quad n\geq 0\] and the lemma is proved.
 	\end{proof}

\begin{lemma}\label{lem:periodic}
	Let $A\in\C(G)$ and let $D$ be a fixed point of $(\C(G),\tilde{f})$. 
	Also, suppose that there is some fixed point $y\in A$.
	Let $x\in D$ be a periodic point.
Then $\lim_n\dist(x,f^n(A))=0$.	
\end{lemma}
\begin{proof}
	Let $p\geq 1$ be a period of $x$ and denote by $B=\{x,f(x),...,f^{p-1}(x)\}$ its orbit.
	If we prove that for some $r\geq 1$ and every $z\in B$, $\lim_n\dist(z,f^{nr}(A))=0$ then, by continuity of $f$ and by the fact that $f$ permutes elements of $B$, we get that for every $z\in B$ and every $0\leq k<r$ $\lim_n\dist(z,f^{nr+k}(A))=0$ which proves the statement.
	So let us take any $z\in B$ and replace $f$ with $f^p$, i.e. suppose $z$ is a fixed point and let us show that $\lim_n\dist(z,f^n(A))=0$.
	
	If, for some $n\geq 0$, $z\in f^n(A)$ then we are done so suppose the opposite, i.e. that $z\notin \bigcup_{n\geq 0}f^n(A)$.
	Specially, for each side $T$ of $z$, $f^n(A)$ does not contain any element of $T$ for any $n\geq 0$.

Take any $\epsilon>0$, denote by $m$ the number of edges in $G$ and set $m^\prime=2m2^m$.
Since $f(z)=z$, there is $\delta>0$ such that $f^i(B(z,\delta))\subset B(z,\epsilon)$ for all $1\leq i\leq m^\prime$.
	By Lemma~\ref{lem:growing_arc}  and since there is obviously some one-sided {neighborhood of $z$} contained in $D$, there is some $t\geq 0$ and $1\leq k\leq m^\prime$ such that $f^{t+nk}(A)\cap B(z,\delta)\neq\emptyset$ for each $n\geq 0$.
 	This implies that for each $n\geq t$, $\dist(z,f^n(A))<\epsilon$, proving the lemma.
\end{proof}

\begin{lemma}\label{lem:alpha_periodic}
	Let $A\in\C(G)$ and let $D$ be a fixed point of $(\C(G),\tilde{f}).$ 
	Also, suppose that there is some fixed point $y\in A$.
	Let $x\in D$ be a  non-periodic point with backward branch $\left(x_j\right)_{j\leq 0}$  in $D$ such that  $\alpha\left(\left(x_j\right)_{j\leq 0}\right)\subset D$ is a periodic orbit.
Then there is $n_0\in\mathbb{N}$ such that $x\in f^n(A)$ for all $n\geq n_0$.
\end{lemma}

\begin{proof}
	The goal is to show that there is some $s\leq 0$ such that $x_s\in f^n(A)$ for all but at most finitely many $n\geq 0$.
	Denote \[\alpha\left(\left(x_j\right)_{j\leq 0}\right)=\{z_0,z_1,...,z_{p-1}\}\] and take some $z\in \alpha((x_j)_{j\leq 0})$.

 Note that for each $0\leq i<p$ there is some $\epsilon_i>0$ such that $x_{j_1},x_{j_2}\in B(z_i,\epsilon_i)$ implies $j_1\equiv j_2\pmod{p}$. 
	Let us renumerate the elements of $\{z_0,z_1,...,z_{p-1}\}$ in a way that \[x_{j}\in B(z_i,\epsilon_i)\quad \text{implies}\quad j\equiv i\pmod{p}.\]
	It is enough to prove that for each $0\leq i<p$ there is some $x_{j_i}\in B(z_i,\epsilon_i)$ and some $n_i$ such that $x_{j_i}\in f^{np}(A)$ for all $n\geq n_i$.
	This would imply that $x\in f^{np+i}(A)$ for all but finitely many $n$.
	Showing that such property holds for all $0\leq i<p$ would prove the lemma.
	This gives us right to assume, without loss of generality, that $\alpha\left(\left(x_j\right)_{j\leq 0}\right)=\{z\}$ for some fixed point $z\in G$, since all the assumptions of the lemma hold if we consider $f^p$ instead of $f$.
Under this assumption, by Lemma~\ref{lem:periodic} we obtain that $\lim_n\dist(z,f^n(A))=0$.	
		
	For $t^\prime$ being the {valence} of $z$, let us choose and numerate the sides $\{T_1,T_2,...,T_{t^\prime}\}$ of $z$ in a way that their elements either contain infinitely many elements of $\left(x_j\right)_{j\leq 0}$ or none of them.  
As a consequence, there is $t,\ 1\leq t\leq t^\prime$ such that the elements of $T_i$ contain infinitely many elements of $\left(x_j\right)_{j\leq 0}$ if and only if $i\leq t$.

%
%

	In what follows, we will present a procedure which will, in finitely many steps, provide us with some $m\geq 0$ and some $k$, $1\leq k\leq t!$, such that, for each $i\leq t$, exactly one of the following three properties holds:
\begin{enumerate}[(a)]	
	\item\label{lem_alph_per_first} there is a $T_i$-sided neighborhood contained in $f^{m+nk}(A)$ for all $n\geq 0$, 
	\item the case \eqref{lem_alph_per_first} does not hold but there is a $T_i$-sided neighborhood $K$ such that, for each $y\in K$, $y\in f^{m+nk}(A)$ for all but at most finitely many $n\geq 0$ or	
	\item $f^{m+nk}(A)$ does not contain any element of $T_i$ for any $n\geq 0$
 and, moreover, for all small enough one-sided neighborhoods $K\in T_i$,  $f^{m+nk}(A)\cap K=\emptyset$ for infinitely many $n\geq 0$.
	\end{enumerate}		

First suppose that there are some $s\geq 0$, some $i\leq t$ and some $K_1\in T_i$ such that\linebreak $K_1\subset f^s(A)$.
	We claim that there are some $j\leq t$, $m_1\geq s$, $k_1\leq t$ and some $L_1\in T_j$ such that $L_1\subset f^{m_1}(A)$ and $\overline{L_1}\subset f^{k_1}(\overline{L_1})$.
	Indeed, denote $\kappa=\dist(z,\Br(G)\setminus\{z\})$ and take $\delta>0$ such that $f^r(B(z,\delta))\subset B(z,\kappa)$ for $0\leq r\leq t!$.
	Since $x_j \to z$ as $j\to-\infty$, there are $j_1<j_2$ such that \[x_{j_1},x_{j_2}\in K_1\cap B(z,\delta)\quad\text{and}\quad\dist(z,x_{j_1})<\dist(z,x_{j_2}).\]
	Let $f^{i_1}(x_{j_1})$ be the first iterate of $x_{j_1}$ not contained in $\overline{zx_{j_1}}$. 
	Then by the construction \[\overline{zf^{i_1}(x_{j_1})}\subset f(\overline{zx_{j_1}})\subset B(z,\kappa).\]
	Now let $f^{i_2}(x_{j_1})$ be the first iterate of $x_{j_1}$ not contained in $\overline{zx_{j_1}}$ nor in $\overline{zf^{i_1}(x_{j_1})}$.
	Then \[\overline{zf^{i_2}(x_{j_1})}\subset f^{s_1}(\overline{zx_{j_1}})\subset B(z,\kappa)\] and \[\overline{zf^{i_2}(x_{j_1})}\subset f^{s_2}(\overline{zf^{i_1}(x_{j_1}}))\subset B(z,\kappa)\] for some $1\leq s_1,s_2\leq 2$.
	We continue the procedure and, since we have finitely many sides and since there is $x_{j_2}\notin \overline{zx_{j_1}}$, after finitely many steps we find $j,m_1,k_1$, where $k_1\leq t$, $L_1\in T_j$  and $L_1\subset f^{m_1}(A)$ such that $\overline{L_1}\subset f^{k_1}(\overline{L_1})$ and hence $L_1\subset f^{m_1+nk_1}(A)$ for all $n\geq 0$.

	We proceed with another side. In that spirit, suppose that there is some $i^\prime\leq t$, $i^\prime\neq j$ and some $K_2\in T_{i^\prime}$ such that $K_2\subset f^{m+k_1s^\prime}(A)\cap B(z,\delta)$ for some $s^\prime$.
 We will now prove that either we can find some $k_2\leq t-1,m_2\geq 0,j^\prime\neq j$ and $L_2\in T_{j^\prime}$ such that
	 $L_2\subset f^{m_1+m_2 k_1}(A)$ and $\overline{L_2}\subset f^{k_1k_2}\left(\overline{L_2}\right)$ 
	 or that for every
	 $w\in K_2$, $w\in f^{m_1+nk_1}(A)$ for all but at most finitely many $n\geq 0$.

 First suppose that, among the sides of $z$, there is a one-sided neighborhood $N$ of $z$ contained in $f^{m_1+nk_1}(A)$ for all $n\geq 0$ with property $\overline{N}\subset f^{k_1}(\overline{N})$ and which contains infinitely many elements of $(x_j)_{j\leq 0}$ with the indexes congruent modulo $k_1$ to indexes of infinitely many elements of $\{x_j\colon j\leq 0\}\cap K_2$.
 	Then, if there is a one-sided neighborhood $K^\prime\subset K_2$ and $v\geq 0$ such that $K^\prime\subset f^{vk_1}(\overline{N})$ then \[K^\prime\subset f^{m_1+nk_1}(A)\quad\text{for all}\quad n\geq v\] and we are done by setting $j^\prime=i^\prime$ and $L_2= K^\prime$.
 	On the other hand, if there is no such $T_{i^\prime}$-sided neighborhood, we easily see that 
 	each $w\in K_2$ is contained in $f^{n{k_1}}(N)\subset f^{m_1+nk_1}(A)$ for all but finitely many $n\geq 0$.	
 	Indeed, first note that $K_2\subset B(z,\kappa)$ so  $K_2\cap N=\emptyset$.
 	For each $w\in K_2$, there are some $v_1,v_2\geq 0$, $v_1<v_2$, congruent modulo $k_1$ such that $x_{v_1}\in N$ and $x_{v_2}\in \overline{zw}$.
 	Since, for all $v\geq0$, $f^{vk_1}(\overline{N})$ does not contain any element of $T_{i^\prime}$, we see that \[w\in f^{m_1+nk_1}(A)\quad\text{for all}\quad n\geq (v_2-v_1)/k_1.\]

	Now we suppose the other case, i.e. that there 
	does not exist any
	one-sided neighborhood $N$ of $z$ contained in $f^{m_1+nk_1}(A)$ for all $n\geq 0$ with property $\overline{N}\subset f^{k_1}(\overline{N})$ and which contains infinitely many elements of $(x_j)_{j\leq 0}$ with the indexes congruent modulo $k_1$ to indexes of infinitely many elements of $\{x_j\colon j\leq 0\}\cap K_2$.
We imitate the procedure we did for the first side considered, i.e. we take some
	 \[x_{j_1^\prime},\ x_{j_2^\prime}\in K_2\subset B(z,\delta),\quad \dist(z,x_{j_1^\prime})<\dist(z,x_{j_2^\prime}),\] such that $j_1^\prime<j_2^\prime$ and $j_1^\prime$ is congruent to $j_2^\prime$ modulo $k_1$.
	 Also, we choose them in a way that only finitely many elements of $(x_j)_{j\leq 0}$ with indexes congruent to $j_1,j_2$ are contained in one-sided neighborhoods of $z$ which lie in $f^{m_1}(A)$ and are self-covering under $f^{k_1}$.
	Furthermore, suppose that we have decreased $\delta$ in a way that $\bigcup_{n=0}^{t}f^{nk_1}(B(z,\delta))$ does not contain those finitely many elements of $(x_j)_{j\leq 0}$ which are described above.
	This way we find some $k_2\leq t-1,m_2,j^\prime\neq j$ and $L_2\in T_{j^\prime}$ such that
	 \[L_2\subset f^{m_1+m_2 k_1}(A)\quad\text{and}\quad\overline{L_2}\subset f^{k_1k_2}\left(\overline{L_2}\right).\]

	 After finitely many steps of this procedure and, if necessary, renumeration of the sides of $z$, we find some $t^{\prime\prime},m\geq 0$ and some $k$, $1\leq k\leq t!$, such that, for $i\leq t^{\prime\prime}$, $f^{m+nk}(A)$ either contains some $L_i\in T_i$ for all $n\geq 0$ or there is some $L_i\in T_i$ such that for all $w\in L_i$, $w\in f^{m+nk}(A)$ for all but at most finitely many $n\geq 0.$ 
	 At the same time, $f^{m+nk}(A)$ does not contain any element of $T_j$ for any $n$ if $t^{\prime\prime}<j\leq t$.

	We can divide the rest of the proof into two cases, depending on whether $t^{\prime\prime}=t$ or\linebreak $t^{\prime\prime}<t$.
	Let us first consider the former, i.e. that there are some $m\geq 0$ and $1\leq k\leq t!$ such that for each side $T_i$ whose elements contain infinitely many elements of the backward branch $(x_j)_{j\leq 0}$ either:
	\begin{enumerate}
	\item $f^{m+nk}(A)$ contains some $L_i\in T_i$ for all $n\geq 0$, or
	\item each point of some $L_i\in T_i$ lies in $f^{m+nk}(A)$ for all but at most finitely many $n\geq 0$.
\end{enumerate}	
	Then there is some $b\leq 0$ and a set $B=\{x_b,x_{b+1},...,x_{b+k-1}\}$ of $k$ consecutive elements of the backward branch such that $B\subset f^{m+nk}(A)$ for all but at most finitely many $n\geq 0$.
	 This implies that $x_{b+k-1}\in f^{n}(A)$ for all except for maybe finitely many $n\geq 0$.
	Therefore, there is some $n_0\geq 0$ such that $x\in f^n(A)$ for all $n\geq n_0$ and we are done with that case.
	
	Now let us consider the case where $t^{\prime\prime}<t$, i.e. $f^{m+kn}(A)$ does not contain any element of the sides $T_i$, $i=t^{\prime\prime}+1,t^{\prime\prime}+2,...,t$  for any $n\geq 0$. We also allow that $t^{\prime\prime}=0$ in this case.
	
	If $t^{\prime\prime}>0$, denote by $\epsilon_1>0$ the minimum over all lengths of $L_i$ for $i=1,2,...,t^{\prime\prime}$. 
	On the other hand, if $t^{\prime\prime}=0$, set $\epsilon_1=\kappa$.
	 Denote by $e$ the number of edges in $G$ and set $l=k(2e2^e)^{t-t^{\prime\prime}}$.
	 We can pick some $\epsilon_2<\epsilon_1$ such that $B(z,\epsilon_1)\setminus B(z,\epsilon_2)$ contains some set $C$ with $l$ consecutive elements of $(x_j)_{j\leq 0}$, 
none of them contained in the elements of sides $T_i$, $i> t$.
	 Let us begin with a side $T_{t^{\prime\prime}+1}$.
By Lemma~\ref{lem:growing_arc}, there are some $t_1\geq 0$ and some $r_1\leq 2ke2^e$ such that \[f^{t_1+nr_1}(A)\cap B(z,\epsilon_2)\cap\left(\cup T_{t^{\prime\prime}+1}\right)\neq\emptyset\quad\text{for all}\quad n\geq 0.\]
	Now we consider $f^{r_1}$ and continue with side $T_{t^{\prime\prime}+2}$, finding some $t_2\geq 0$, $r_2\leq k(2e2^e)^2$ such that \[f^{t_2+nr_2}(A)\cap B(z,\epsilon_2)\cap\left(\cup T_{t^{\prime\prime}+1}\right)\neq\emptyset\] and \[f^{t_2+nr_2}(A)\cap B(z,\epsilon_2)\cap\left(\cup T_{t^{\prime\prime}+2}\right)\neq\emptyset\] for all $n\geq 0$.
	We proceed by considering $f^{r_1r_2}$.
	After finitely many steps we find some $c\geq 0$ and some $1\leq d\leq l$ such that $C\subset f^{c+nd}(A)$ for all $n\geq 0$.
	{Since $d\leq l$ and since $C=\{x_{b_1},x_{b_1+1},\ldots,x_{b_1+l-1}\}$ for some $b_1<0$, we have that $x_{b_1+l-1}\in f^{c+nd+w}(A)$ for all $n\geq 0$ and all $w$ such that $0\leq w\leq d-1$.
 We conclude that $f^n(A)$ contains $x$ for all but at most finitely many $n\geq 0$.}	  
	\end{proof}

In what follows, we will state Lemma~\ref{lem:alpha_basic} which covers the case of points whose alpha limit set is an infinite subset of a basic set.
	In its proof we are going to use already known results from the following two auxiliary lemmas.
	As first we state Lemma~\ref{lem:modelBS} which	provides us with a model map which is often used when dealing with basic sets.
	The reader is referred to \cite{B1,B2} for more details, see also \cite{Snoha,Forys}. 	
	
\begin{lemma}\label{lem:modelBS}
	Let $f\colon G\rightarrow G$ be a graph map and $X\subseteq G$ be a cycle of graphs. Suppose that $D(X)$ is a basic set. 
	Then there is a transitive map $g\colon Y\rightarrow Y$, where $Y$ is a cycle of graphs $Y_0,\ldots Y_{n-1}$ with possibly non-empty intersection in the endpoints, and $\phi\colon X\rightarrow Y$ which almost conjugates $f|_{D(X)}$ and $g$.
	 Moreover, $g^n|Y_i$ is mixing, for $i=0,\ldots,n-1$.
	  The period $n$ of $Y$ is a multiple of the period of $X$ and $Y_i\cap Y_j=\End(Y_i)\cap \End(Y_j)\neq \emptyset$ iff $i\neq j$ and $i$ and $j$ are congruent modulo the period of $X$.	
\end{lemma}

	The proof of Lemma~\ref{lem:basicset-sub} which comes next as well as of the following Corollary \ref{cor:basicset-sub} can be found in the preceding paper (~\cite[Lemma 3.5, Corollary 3.6]{Jelic}).

\begin{lemma}\label{lem:basicset-sub} 
	Assume that a continuum $A$ is recurrent and $\omega=\omega_f(x)$ 
	is a basic set for some $x\in G$. 
	Denote by $K_1,K_2,...,K_m$ minimal (in the sense of inclusion) graphs such that $\omega\subset \bigcup_i K_i$, $f$ permutes graphs $K_i$
	and $m$ is such that $f^m|_{K_i}$ is almost conjugated to a mixing map.
	If $\Int A\cap \omega \cap K_i\neq\emptyset$ then $K_i\subset A$.
\end{lemma}	

\begin{corollary}\label{cor:basicset-sub} 

	 Assume that a nondegenerate continuum $A$ is recurrent and $\omega=\omega_f(x)$ 
is a basic set for some $x\in G$.
	Denote by $K_1,K_2,...,K_r$ {the sugraphs provided by Lemma~\ref{lem:modelBS}, i.e.} minimal (in the sense of inclusion) graphs such that $\omega\subset \bigcup_i K_i$, $f$ permutes graphs $K_i$
	and $r$ is such that $f^r|_{K_i}$ is almost conjugated to a mixing map.
	If $A\cap\omega\cap K_i\neq\emptyset$ then either $K_i\subset A$
or  $A\cap\omega\cap K_i$ is finite set and each $y\in A\cap\omega\cap K_i$ is eventually periodic and 
moreover, $A\cap K_i\cap \omega$ contains a periodic point from orbit of each such $y.$
\end{corollary}

\begin{lemma}\label{lem:alpha_basic}
Let $A\in\C(G)$ and
let $D$ be a nondegenerate fixed point of $(\C(G),\tilde{f})$. 
Let $w\in D$ be a point with backward branch $\{w_j\}_{j\leq 0}$  in $D$ such that $\alpha(\{w_j\}_{j\leq 0})$ is infinite and contained in a basic set  $\omega=E(K,f)\subset D$.
Then $\lim_n\dist(w,f^n(A))=0$.	
\end{lemma}		

\begin{proof}
{By assumptions $\omega$ is a basic set and $D$ is an invariant (hence recurrent) continuum, so we can apply Lemma~\ref{lem:basicset-sub}. Denote by}
$K_1,K_2,...,K_m$ minimal (in the sense of inclusion) graphs such that $\omega\subset \bigcup_i K_i{ =K}$, $f$ permutes graphs $K_i$
	and $m$ is such that $f^m|_{K_i}$ is almost conjugated to a mixing map.
	Take any $0\leq i\leq m$.
{Note that $K\subset D$ since $f(D)=D$ and $\omega\subset D$.} 

Since $D\in \omega_{\tilde{f^m}}(A)$ and $\omega$ is perfect{ as a basic set (e.g. see ~\cite[Theorem 2]{B1})}, there is some $k\geq 0$ such that $\omega\cap f^{km}(A)\cap K_i$ is infinite.
	Since $\omega$ is perfect and contained in $\overline{\Per(f)}$, there are 
	periodic points 
	\[x_1,x_2,y_1,y_2\in f^{km}(A)\] of common (not necessarily minimal) period $p$
	contained in the interior of the same edge $E$ of $D$ such that \[\overline{x_1 x_2}\subset f^{km}(A) \cap K_i,\] $\overline{x_1 x_2}\cap\omega\cap K_i$ is infinite,
	{\[\overline{y_1 y_2}\subset f^{km}(A) \cap K_i,\] $\overline{y_1 y_2}\cap\omega\cap K_i$ is infinite, \[\overline{x_1x_2}\cap \overline{y_1y_2}=\emptyset\] and each connected component of  $E\setminus(\overline{x_1x_2}\cup\overline{y_1y_2})$ has infinite intersection with $\omega\cap K_i$.} 
Denote $L_0=\overline{x_1 x_2}$ and $L_1=\overline{y_1 y_2}$. Then either $f^p(L_0)\supset L_0$ or $f^p(L_1)\supset L_1$ or
$f^p(L_0)\supset L_1$ and $f^p(L_1)\supset L_0$.
Therefore we can find some $r\in\{0,1\}$, $t\in\{1,2\}$ such that 
\[L_r\subset f^{pt}(L_r),\quad \Int L_r\cap\omega\cap K_i\neq\emptyset\quad \text{and}\quad 
L_r\subset  f^{km}(A).\] 
 Note that $K_i\cap f^{pt}(K_i)$ is infinite and therefore $m$ is a divisor of $pt$.
	
	Denote 
 $L=\overline{\bigcup_{j=0}^\infty f^{jpt}(L_r)}$	 and observe that
$L$ is a periodic continuum intersecting $\Int K_i$, so $K_i\subset L$.
This means that for each $y\in K_i$,
$\lim_n\dist(y,f^{km+npt}(A))=0$.
Since $f^m(K_i)=K_i$, we easily see that for each $y\in K_i$, 
$\lim_n\dist(y,f^{km+
mn}(A))=0$ and hence $\lim_n\dist(y,f^{mn}(A))=0$.	
Since $0\leq i\leq m$ was arbitrary, we proved that for each $j$, $0\leq j\leq m$, and all $y\in K_j$, $\lim_n\dist(y,f^{mn}(A))=0$.	
Since $f(\cup_i K_i)=\cup_i K_i$, it follows that for all $y\in \cup_i K_i$, $\lim_n\dist(y,f^{n}(A))=0$.	 	
		
	But $(w_j)_{j\leq 0}$ is a sequence such that $\alpha((w_j)_{j\leq 0})$ is infinite and is contained in $\cup_i K_i$. Using the fact that $\cup_i K_i$ is invariant, for each $j\leq 0$ we have $w_j\in  \cup_i K_i$.
	Therefore, $\lim_n\dist(w,f^n(A))=0$ which proves the lemma. 
\end{proof}


\begin{lemma}\label{lem:subset_of_D}
	{Let $A\in\C(G)$ and
let $D$ be a nondegenerate fixed point of $(\C(G),\tilde{f})$ and a unique periodic element of $\omega_{\tilde{f^n}}(A)$ for every $n\geq 1$.
	Then there are some $m\geq 0$ and $t\geq 1$ such that
 $\limsup_n f^{m+tn}(A)\subset D$. Furthermore, if $f^{m+tn}(A)\setminus D\neq \emptyset$ for all $n\geq 0$ then $\left(\bigcup_{n\geq 0} f^{m+tn}(A)\right)\cap \Per(f)\neq\emptyset$.}
\end{lemma}
\begin{proof}

{To see this, first observe that if $f^n(A)\subset D$ for some $n\geq 0$ then we are done with the proof of both claims by putting $m=n$ and $t=1$.}
	Now suppose $f^n(A)\setminus D\neq\emptyset$ for all $n\geq 0$.
	Take $\epsilon>0$ such that $N(D,\epsilon)\setminus D$ does not contain any branching point.
	We can pick a sequence $(n_k)_{k\geq 1}$ of positive integers such that $f^{n_k}(A)\subset N(D,\epsilon/k)$ and $f^{n_k}(A)\cap D\neq\emptyset$  for all $k\geq 1$.
	Note that also $f^{n_k}(A)\cap D$ is connected for all $k\geq 1$.
	{Since there are finitely many edges in $G$, we may assume that, for any two $k_1,k_2>0$, $f^{n_{k_1}}(A)\setminus D$ and $f^{n_{k_2}}(A)\setminus D$ intersect the same edges of $G$.
    Now we easily see that some $k_1<k_2$ satisfy $f^{n_{k_2}}(A)\setminus D\subset f^{n_{k_1}}(A)\setminus D$.}
	$D$ is invariant so, for $t=n_{k_2}-n_{k_1}$ and $m=n_{k_1}$, we have 
	\begin{equation}\label{eq:nested}
	f^{m+(n+1)t}(A)\setminus D\subset f^{m+nt}(A)\setminus D \subset f^{m}(A)\setminus D\text{ for all } n\geq 0.
	\end{equation}
{But $D\in\omega_{\tilde{f}^{t}}(f^{m}(A))$} and, therefore, $\limsup_n f^{m+nt}(A) \subset D$. 

We have yet to prove that $\left(\bigcup_{n\geq 0} f^{m+tn}(A)\right)\cap \Per(f)\neq\emptyset$.
Recall that there is some $r\geq 1$ such that, for all $n\geq 0$, $\overline{f^{m+nt}(A)\setminus D}=\cup_{i=1}^r L^n_i$, where $\{L^n_i\colon 1\leq i\leq r\}$ are pairwise disjoint arcs such that, for each $1\leq i\leq r$, $\{L_i^n\colon n\geq 0\}$ is a decresing nested set of arcs with $\cap_{n\geq 0} L^n_i=\{l_i\}$ for some $l_i\in D$. Let $s>0$ be the cardinality of $\{l_i\colon 1\leq i\leq r\}$ (they may coincide for different indices) and denote $C=\{c_i\colon 1\leq i\leq s\}=\{l_i\colon 1\leq i\leq r\}$.
We claim that $f^t(C)=C$.
Take some $\delta>0$ such that, for all $i\leq s$, $f^t(B(c_i,\delta))\cap C$ is an empty set when $f^t(c_i)\notin C$ and a singleton when $f^t(c_i)\in C$.
Let $n_1\geq 0$ be such that $f^{m+n_1t}(A)\subset N(D,\delta)$. But $f^{m+n_1t}(A)\setminus D$ and $f^{m+{(n_1+1)t}}(A)\setminus D$ intersect the same edges of $G$.
Note that \[f^{m+{(n_1+1)t}}(A)\setminus D\subset f^{m+n_1t}(A)\setminus D\subset\bigcup_{i=1}^s B(c_i,\delta).\]
Also, $D$ is invariant, so the set $\overline{f^{m+{(n_1+1)t}}(A)\setminus D}$ is covered by the sets of type $f^t(B(c_i,\delta))$ and, by the choice of $\delta$, indeed $f^t(C)=C$. But $C$ is finite and, therefore, $f^t$ permutes its elements. 
Since $C\subset f^m(A)$, the lemma is proved.


\end{proof}




The following lemma encapsulates the series of auxiliary results presented above in the case when there is a periodic point $y\in f^m(A)$ for some $m\geq 0$.
 
	\begin{lemma}\label{lem:fixed_D_fixed_y}
Let $A\in\C(G)$ and
let $D$ be a nondegenerate fixed point of $(\C(G),\tilde{f})$. 
Moreover, suppose that there is a periodic point $y\in f^m(A)$ for some $m\geq 0$.
Then $\omega_{\tilde{f}}(A)=\{D\}$.
\end{lemma}
\begin{proof}
 {By Lemma~\ref{lem:subset_of_D} and {Lemma~\ref{lem:iterates}}, we can assume that $\limsup_n f^n(A)=D$, that $m=0$ and that $y$ is a fixed point.}

Take any point $x\in D$ and some $(x_j)_{j\leq 0}$, a backward branch of $x$ contained in $D$.
Note that this is possible since $f\vert_D\colon D\to D$ is onto.
	A least one of the following holds:
\begin{enumerate}
\item\label{case:first} $x$ is a periodic point,
\item  $x$ is a non-periodic point and $\alpha((x_j)_{j\leq 0})$ is a periodic orbit,
\item $\alpha((x_j)_{j\leq 0})$ is a minimal solenoid,
\item $\alpha((x_j)_{j\leq 0})$ is an infinite subset of a basic set, or
\item\label{case:fifth} $\alpha((x_j)_{j\leq 0})$ is a circumferential set.	
\end{enumerate}		
Therefore, by Lemmas~\ref{lem:periodic},~\ref{lem:alpha_periodic},~\ref{lem:alpha_solenoid},~\ref{lem:alpha_basic} and Corollary~\ref{cor:alpha_circumferential}, respectively, we have $\lim_n\dist(f^n(A),x)=0$ in all cases.
	
Indeed, $\omega_{\tilde{f}}(A)=\{D\}$ and the proof is completed. 
\end{proof}

	The following theorem is a result of~\cite{MaiRecurrent} which will be of use in the proof of Lemma~\ref{lem:no_periodic}.

{\begin{theorem}\label{thm:recurrent_pts}
Let $f\colon G\to G$ be a graph map and $L=\Int \overline{xy}$ be an open arc contained in an edge of $G$.
If there exist $m,k\in\mathbb{N}$ such that $\{f^m(x),f^k(y)\}\subset L$, then 
$\Rec(f)\cap L\neq\emptyset$.
\end{theorem}
}

\begin{lemma}\label{lem:no_periodic}
Let $A\in\C(G)$ and
let $D$ be a nondegenerate fixed point of $(\C(G),\tilde{f})$ and a unique periodic element of $\omega_{\tilde{f^n}}(A)$ for every $n\geq 1$.
Moreover, suppose that $\left(\bigcup_{n=0}^\infty f^n(A)\right)\cap\Per(f)=\emptyset$.
Then $\omega_{\tilde{f}}(A)=\{D\}$.
\end{lemma}	
\begin{proof}
 {Since $\left(\bigcup_{n=0}^\infty f^n(A)\right)\cap\Per(f)=\emptyset$, from Lemma~\ref{lem:subset_of_D} it directly follows that\linebreak $\limsup_n f^{m+nt}(A)\subset D$ for some $m\geq 0$ and $t\geq 1$ and that $f^{m+tn}(A)\subset D$ for some $n\geq 0$.
 Now by Lemma~\ref{lem:iterates}, we may assume that $m=n=0$ and $t=1$.
 In particular, $A\subset D$.}
	First note that $D\in\omega_{\tilde{f}}(A)$ implies that $A$ is non-wandering{, i.e. $f^t(A)\cap f^r(A)\neq \emptyset$ for some $t\neq r$}.
	If there are no circumferential $\omega$-limit sets in $D$ then, by Lemma~\ref{lem:per_pts}, there exists some finite set $B$, a union of periodic orbits, such that $\omega_f(x)\subset B$ for all $x\in A$. 
 Note that $B$ does not necessarily equal whole set $\Per(f)\cap D$.
However, since $D\in \omega_{\tilde{f}}(A)$, there can be at most finitely many periodic points in $D\setminus\left(\bigcup_{n=0}^\infty f^n(A)\right) $ so we can extend $B$ to whole $\Per(f)\cap D$	
	and then collapse it to a point. Now the map  $f/_B\colon D/_B\to D/_B$ has a unique fixed point. Condition $\left(\bigcup_{n=0}^\infty f^n(A)\right)\cap\Per(f)=\emptyset$ also implies that  $D$ and as a consequence also  $D/_B$ do not contain solenoids or basic sets. 
	Hence the unique fixed point $z$ is also a unique $\omega$-limit point of  restriction of $f/_B$ to $D/_B$. But then also for any backward branch  contained in $D$, $\{z\}$ is its $\alpha$-limit set. Then Lemma~\ref{lem:alpha_periodic} implies that 
		$\lim_n\dist(y,f^n(A))=0$ for any $y\in D$ provided that $A\cap B\neq \emptyset$, i.e. $A/_B$ contains the unique periodic point.
	
{We have to consider remaining two cases. Assume first }	
that $D$ contains some circumferential sets.
	There can be at most finitely many of them so denote them by $E(K_i,f)$, where $K_i,\ 1\leq i\leq v$, are $v$ pairwise disjoint minimal cycles of graphs containing them.
	By Lemma~\ref{lem:contains_circumferential}, for each $y\in\bigcup_{i=1}^v K_i$ we have  $\lim_n\dist(y,f^n(A))=0$.
	Since $L=\bigcup_{i=1}^v K_i$ is a closed set with finitely many components and is strongly invariant, we can collapse it to a point, i.e. we continue by considering the map $f/_L\colon D/_L\to D/_L$.
 Note that $D/_L$ is now the unique periodic element of $\omega_{\tilde{f}^n/_L}(A/_L)$ for every $n\geq 1$. 
	But also in that case, $\bigcup_{n=0}^\infty(f/_L)^n(A/_L)$ contains a periodic point and, by  Lemma~\ref{lem:fixed_D_fixed_y}, we have $\omega_{\tilde{f}/_L}(A/_L)=\{D/_L\}$.
	This, together with  the fact that for all $y\in L$, $\lim_n\dist(y,f^n(A))=0$, implies $\omega_{\tilde{f}}(A)=\{D\}$ and the lemma is proved in this specific case.

	
{Finally, as the last case assume that }there is a fixed point $z\in D$ such that $\omega_f(x)=\{z\}$ for all 
{$x\in D$}
and $z\notin\left(\bigcup_{n=0}^\infty f^n(A)\right)$.
	
	First, let us show that for each $m\geq 1$ there exists some $\delta_m>0$ such that for each one-sided neigborhood $L$ of $z$, contained in $B(z,\delta_m)$, its images $f^i(L),\ 1\leq i\leq m$ are again one-sided neighborhoods of $z$.
	To that end, take some $m\geq 1$. 
	Since $z$ is a fixed point, we can pick $\delta_m>0$, such that \[\bigcup_{i=0}^m f^i(B(z,\delta_m))\cap \Br(D)\setminus\{z\}=\emptyset.\]
Suppose that for some $i\leq m$ and some one-sided neighborhood $L\subset B(z,\delta_m)$ of $z$, $f^i(L)$ intersects more than one one-sided neighborhood of $z$.
Since $L\subset B(z,\delta_m)$, it follows that $z\in f^i(L)$.
	Take some $y\in L$ such that $f^i(y)=z$.
	Since $D\in\omega_{\tilde{f}}(A)$, there is some $s\geq 0$ such that \[f^s(A)\setminus L\neq\emptyset\quad\text{and}\quad\diam (f^s(A)\cap L)>\diam (L)-\dist(y,z).\]
	Since $z\notin f^s(A)$, necessarily $y\in f^s(A)$ but then $z\in f^{s+1}(A)$ which is a contradiction.
	
	Denote by $v$ the valence of $z$. 
	Take any $x\in A$ and $\delta_v>0$ provided by construction from the previous paragraph.
Since $\omega_f(x)=\{z\}$, there  is some $n_0$ such that $f^n(x)\in B(z,\delta_v)$ for $n\geq n_0$.
	But then the set \[\{f^{n_0}(x),f^{n_0+1}(x),...,f^{n_0+v}(x)\}\] contains at least two elements $f^{s_1}(x),f^{s_1+s_2}(x)$ from the same one-sided neighborhood $L$ of $z$ with $\diam L=\delta_v$.
	By the previous argument it follows that $f^{s_1+ns_2}(x)\in L$ for all $n\geq 0$.
By Lemma~\ref{lem:iterates}, we can suppose that that $f^n(x)\in L$ for all $n\geq 0$.
Let $\epsilon=\dist(x,z)$. There is some $n_1\geq 0$ such that $f^n(x)\in B(z,\epsilon)\cap L$ for all $n\geq n_1$.
 Note that from the assumption it follows that $\omega(f)=\{z\}$ and hence \[(\Rec(f)\cap L)\subset(\omega(f)\cap L)=\emptyset.\]	
 This, together with Theorem~\ref{thm:recurrent_pts}, implies that $(f^{n_1+n}(x))_{n\geq 0}$ is a strictly monotone sequence in the ordering of $L$. 
 	Similarly to the proof of Lemma~\ref{lem:periodic}, let us denote $m^\prime=2m2^m$, where $m$ is the number of edges in $D$.
 	Let $\delta,\ 0<\delta<\epsilon<\delta_v$ be such that $\bigcup_{i=0}^{m^\prime} f^i(B(z,\delta))\subset  B(z,\epsilon)$.
 	There is some $n_2\geq 0$ such that $f^n(x)\in L\cap B(z,\delta)$ for all $n\geq n_2$.    
 	Let us set $y=f^{n_2}(x)$.
The sequence $(f^n(y))_{n\geq 0}$ is a strictly monotone sequence in the ordering of $L\cap B(z,\delta)$ converging towards $z$.
	
	If $\lim_k\diam(f^{n_k}(A))=0$ for some sequence $(n_k)_{k\geq 0}$ then
	 $f^{n_k}(A)\to\{z\}$. {But then Lemma~\ref{lem:d-not-contained} implies that either $A$ is wandering (so $\lim_k\diam(f^{k}(A))=0$) and then 
	 $\omega_{\tilde{f}}(A)=\{\{z\}\}$
	  or
	 	directly 
$\omega_{\tilde{f}}(A)=\{\{z\}\}$
	 	 (the case of circumferential set in  Lemma~\ref{lem:d-not-contained} is excluded). 
	 	This is a contradiction, so let us assume on the contrary}
%
that there is some $\kappa$, $0<\kappa<\delta$ such that $\diam f^n(A)\geq\kappa$ for all $n\geq 0$.
	This implies that for all $u\in L\cap B(z,\kappa)$, $u\in f^n(A)$ for all but at most finitely many $n\geq 0$.
	
	Take any $w\in D \setminus\{z\}$ and any backward branch $(w_i)_{i\leq 0}$ of $w$  in $D$.
	Again, such backward branch exists since $f\vert_D\colon D\to D$ is onto.
	Note that $\alpha((w_i)_{i\leq 0})=\{z\}$.
	We will show that there is a point $w_q$ for some $q\leq 0$ such that $w_q\in f^n(A)$ for all but finitely many $n$ which will prove the lemma.
	If there is an element of $(w_i)_{i\leq 0}$ contained in $L\cap B(z,\kappa)$ then we are done so suppose that there are some other sides of $z$ whose elements contain infinitely many elements of $(w_i)_{i\leq 0}$.
	Take any such side $T$ and some $N\in T$.

{Replacing $A$ by its iterate, we may assume that $y\in A$.}
	Similarly to the proof of Lemma~\ref{lem:alpha_periodic}, take some $\delta^\prime,\ 0<\delta^\prime<\delta$ such that $B(z,\delta)\setminus B(z,\delta^\prime)$ contains a sequence of  $(m^\prime)^v$ consecutive elements of $(w_j)_{j\geq 0}$.
	There is a point $y^\prime$ in $B(z,\delta^\prime)\cap N$ such that $y^\prime\notin \bigcup_{n=0}^{m^\prime} f^n(A)$.
	Then there is some $r>m^\prime$ such that $y^\prime\in f^r(A)\setminus(\bigcup_{i=0}^{r-1}f^i(A))$ and $y^\prime_0\in A$ such that $f^r(y^\prime_0)=y^\prime$.

 Let us firstly prove that $f^i(y_0^\prime)\not\in \overline{zy}$ for $0\leq i\leq r$.
	 By the construction, if $f^j(y^\prime_0)\in \overline{zy}$ for some $j<r$ then $f^j(y^\prime_0)\in L\cap B(z,\delta)$. 
    Let $j^\prime>j$ be the minimal integer such that\linebreak $f^{j^\prime}(y_0^\prime)\notin \overline{zy}$.
  By the construction \[\overline{zy}\subset B(z,\delta),\quad \bigcup_{i=0}^{m^\prime} f^i(B(z,\delta))\subset  B(z,\epsilon)\quad\text{and} \quad\epsilon=\dist(z,x),\] therefore $f^{j^\prime}(y_0^\prime)\in \Int \overline{f^j(y^\prime_0)x}$.	 		But $y$ is an element in the orbit of $x$ so, by Theorem~\ref{thm:recurrent_pts}, $\omega(f)\cap \overline{f^j(y_0^\prime)x}\neq\emptyset$ {which cannot hold, since in the considered case $\omega(f)=\{z\}$.
  	This implies that 
  	{$f^j(y^\prime_0)\in \overline{zy}\subset L$. }
  	}
	 But we picked $y^\prime\in N$ and $N\cap L=\emptyset$, which is a contradiction.
	Indeed, \[f^i(y_0^\prime)\not\in \overline{zy}\quad\text{for}\quad 0\leq i\leq r.\]	
	
If we pick a family of arcs $L_i,\ 0\leq i\leq m^\prime$ such that each $L_i$ has $f^i(y)$ and $f^i(y_0)$ as endpoints, $L_0\subset A$ and such that $L_{i+1}\subset f(L_i)$ then, like in Lemma~\ref{lem:alpha_periodic}, we can find $t_1,t_2\leq m^\prime$ such that $L_{t_1}\subset f^{t_2}(L_{t_1})$. {Hence}
after taking few more iterates we find a point $y_N\in B(z,\delta^\prime)\cap N$ such that $y_N\in f^{j+nt_2}(A)$ for some $j\geq 0$ and all $n\geq 0$.
	We continue by considering sequence  $(f^{j+nt_2}(A))_{n\geq 0}$ and repeating the same procedure for another side of $z$  {(and the map $f^{t_2}$)} containing infinitely many elements of $(w_i)_{i\leq 0}$.
	After finitely many steps we find some $k\geq 0$ and some $t\leq (m^\prime)^v$ such that $B(z,\delta)\setminus B(z,\delta^\prime)\subset f^{k+nt}(A)$ for all $n\geq 0$.
	Since $B(z,\delta)\setminus B(z,\delta^\prime)$ contains $(m^\prime)^v$ consecutive elements of $(w_i)_{i\leq 0}$, we indeed found some $w_q$ contained in $f^n(A)$ for all but finitely many $n\geq 0$ and the lemma is proved.
\end{proof}

\begin{lemma}\label{lem:periodic_supset}
Let $A\in\C(G)$ and let $D\in\Rec(\tilde{f})$ be provided by Corollary~\ref{cor:proximal}.
	Suppose that $D$ is either nondegenerate or that it is a periodic singleton contained in $f^m(A)$ for some $m\geq 0$.
	Then $A$ either satisfies condition~\eqref{circ} or~\eqref{asy_per} of Theorem~\ref{thm:main}. 
\end{lemma}
\begin{proof}
	From Theorem~\ref{thm:rec-cont} it follows that either $D\in\Per(\tilde{f})$ or $D\in\mathcal{R}$.
	First suppose that $D\in\mathcal{R}$.	
	It is easily checked that, for each $x\in A$, $\omega_f(x)\cap D\neq\emptyset$.
	It follows that $A$ satisfies the condition of Lemma~\ref{lem:circ} and therefore $A$ satisfies the condition~\eqref{circ} of Theorem~\ref{thm:main} and we are done in this case.
	
	Next, let us consider the remaining case, i.e. $D\in\Per(\tilde{f})$.
	We will show that $A$ is asymptotically periodic. 

	Suppose that $p>0$ is the minimal period of $D$ under $\tilde{f}$.
	There is $\epsilon>0$ such that $\dist_H(D,f^n(D))<\epsilon$ implies $D=f^n(D)$ and hence $p\vert n$.
This implies that  $p$ is a divisor of all but at most finitely many elements of $(n_k)_{k\geq 0}$ provided for $D$ and $A$ by Corollary~\ref{cor:proximal}. 
	By this observation and by Lemma~\ref{lem:iterates} we can assume, by replacing $f$ with $f^p$, that $D$ is a fixed point of $\tilde{f}$.
	Furthermore, since $\lim_k\dist_H(f^{n_k}(A),D)=0$, when $D$ is nondegenerate we may also also assume that $f^{n_k}(A)\cap D\neq\emptyset$ for all $k\geq 1$.
On the other hand, when $D=\{d\}$ is a fixed singleton contained in $f^m(A)$ for some $m\geq 0$ then $d\in f^n(A)$ for all $n\geq m$ and again by Lemma~\ref{lem:iterates} we can suppose that $d\in f^n(A)$ for all $n\geq 0$.

	In what follows, we will prove that if, for some $x\in G$, $\{n\geq 0\colon x\in f^n(A)\}$ is infinite then $x\in D$.
	To do this, suppose there is some $x\in G\setminus D$ with this property.
 If $D$ is degenerate then let $\epsilon<1/2\dist_H(\Br(G)\setminus D,D)$, otherwise let
	$\epsilon>0$ be the minimum of the lengths of components of $\Int E\cap D$ and $E\setminus D$ over all edges $E$ of $G$, where we take into account only the positive values among them.
	Let $m$ be the number of edges in $G$ and denote $\delta_{2m}=1/2\min\{\dist(D,x),\epsilon\}$.
	Since $D$ is fixed and $\tilde{f}$ is continuous, we may find some $\delta_i>0$ for $i=0,1,...,2m-1$ such that \[\tilde{f}(B_{\dist_H}(D,\delta_{i}))\subset B_{\dist_H}(D,\delta_{i+1})\quad\text{for}\quad i=0,1,...,2m-1.\]
	Observe that this implies $f(N(D,\delta_{i}))\subset N(D,\delta_{i+1})$, for $i=0,1,...,2m-1$.
	Recall that, by Lemma~\ref{lem:iterates}, we may assume that $A\in B_{\dist_H}(D,\delta_{0})$.

\begin{figure}[H]
\includegraphics[scale=2]{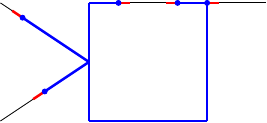}
{\caption{The sketch represents graph $D$ in blue with its boundary in $G$ indicated. Set $N(D,\delta_{2m})\setminus D$ is colored red and its closure contains each arc $L_i$, $i=1,2,\ldots,2m$.}\label{fig:N_D_delta}}
\end{figure}
    
	By assumption, there is some $x_0\in A$, eventually mapped to $x$, say $f^r(x_0)=x$.
	Note that $D$ is fixed by $\tilde f$, so $x_0\notin D$.	
By the definition of $\delta_{2m}$, there is unique arc $L_0$, such that $\Int L_0\subset N(D,\delta_0)\setminus D$ and whose one of the endpoints is $x_0$ and the other is in $D$ {(see Figure~\ref{fig:N_D_delta})}. Note that $L_0\subset A$, {because $x_0\in A$,\linebreak $A \subset N(D,\delta_0), A\cap D\neq\emptyset$ and $A$ is connected}.
Denote $x_j=f^j(x_0)$.
	By assumption there exist some $k_{1}<k_{2}<...<k_{2m}$ such that \[x_{k_i}\in N(D,\delta_{i})\setminus N(D,\delta_{i-1})\] {by the definition of each $\delta_i$, 
    since $x_0\in N(D,\delta_0)$, $f(N(D,\delta_{i}))\subset N(D,\delta_{i+1})$ and $x\notin N(D,\delta_{2m})$}.
    Furthermore, assume each $k_i$ is the minimal such index.
	Analogously to $L_0$ we define $L_{k_i}$ for\linebreak $i=1,2,...,2m$.
	It is easy to see that necessarily $f^{k_{i+1}-k_i}(L_{k_i})\supset L_{k_{i+1}}$.
	Indeed, if $j<r$, $x_j\in N(D,\delta_{i-1})\setminus D$ and $L$ is the unique arc contained in $N(D,\delta_{i})$ for some $i=0,1,2,...,2m-1$ and connecting $D$ with $x_j$ then $x_{j+1}\in N(D,\delta_{i+1})\setminus D$ and $f(L)$ contains the unique arc $J$ in $N(D,\delta_{i+1})$ connecting $D$ with $x_{j+1}$ (recall that $D$ is fixed, so $f(L)\cap D\neq\emptyset$).
	Therefore, there must be some $m\geq 0$ and some $k>0$ such that $f^k(L_{m})\supset L_m$.
	But on the other hand $L_{k_m}\subset f^{k_m}(L_0)\subset f^{k_m}(A)$.
	This contradicts Lemma~\ref{lem:self-covering}. We obtained that for every $\epsilon>0$ there is $n_0$ such that $f^n(A)\subset  N(D,\epsilon)$ for every $n>n_0$.
 { This shows that $A$ is asymptotically periodic in the case of $D$ being a singleton.
Now} either Lemma~\ref{lem:fixed_D_fixed_y} or~\ref{lem:no_periodic} finishes the proof.
\end{proof}	
	
\bigskip
	
\begin{proof}[Proof of Theorem~\ref{thm:main}]
	To prove the theorem, take some $A\in \C(G)$ and some $D\in \Rec(\tilde{f})$ provided by Corollary~\ref{cor:proximal}.
	Let us suppose that conditions~\eqref{wandering} and~\eqref{circ} do not hold.
	We will prove that $A$ is asymptotically periodic.

First suppose that $D=\{d\}$ is a singleton.	
Lemma~\ref{lem:degenerate_d} then implies $d\in\Per(f)$.
	If $d\notin\bigcup_{n=0}^\infty f^n(A)$ then we are done by Lemma~\ref{lem:d-not-contained}.

The two remaining cases are that $D$ is a periodic degenerate continuum contained in $\bigcup_{n=0}^\infty f^n(A)$ or that it is nondegenerate.
	But both those cases are covered by Lemma~\ref{lem:periodic_supset}, therefore the theorem is proved.	
\end{proof}	

\section{Bound on periods}\label{sec:periods}

\begin{lemma}\label{lem:omega_singleton}
	Let $T$ be a tree, $f$ continuous and surjective selfmap of $T$ and let there be a fixed point $z\in T$ such that $\omega_f(x)=\{z\}$ for all $x\in T$.
	Then $T=\{z\}$. 
\end{lemma}

\begin{proof}
Suppose on the contrary, i.e. that $T$ is nondegenerate.	
	Take some $\epsilon>0$ such that $B(z,\epsilon)\cap \V(T)\setminus\{z\}=\emptyset$.
	Since for each vertex $v$ of $T$, $f^n(v)\in B(z,\epsilon)$ for all but at most finitely many $n\geq 0$, we can find some $r\geq 0$ such that $f^n(\V(T))\subset B(z,\epsilon)$ for all $n\geq r$.	
	Denote the points of $\End(T)\setminus\{z\}$ by $\{e_1,e_2,...,e_t\}$ and, for each $i,\ 1\leq i\leq t$, let $L_i=\overline{ze_i}\setminus\{z,e_i\}$.
	Fix any $n\geq r$.
	Similarly as Matviichuk did in~\cite{Matviichuk}, we partition each $L_i$ into two sets, $L_i=M_i{\cup} N_i$, where $M_i=\{x\in L_i\colon x\notin \overline{f^n(x)e_i}\}$	and $N_i=L_i\setminus M_i$. {Since $f^n(\V(T))\subset B(z,\epsilon)$, we see that $N_i\neq \emptyset$. Clearly $M_i$ is open. Now assume that $x\in \overline{f^n(x)e_i}$ and note that by assumptions $x\neq f^n(x)$. If $x\notin V(T)$ then there is $\delta>0$ such that $B(x,\delta)\subset \overline{f^n(x)e_i}$. But if $x\in V(T)$ then, since $f^n(\V(T))\subset B(z,\epsilon)$, $B(x,\delta)\cap L_i\subset \overline{f^n(x)e_i}$. This shows that $N_i$ is open.		
	But sets $N_i,M_i$ are disjoint and $L_i$ is connected, so $M_i=\emptyset$.}
	This implies that $e_i\notin f^n(\overline{ze_i})$ for every $n\geq r$.
	
	Since $f^r(T)=T$ and $f^r(V(T))\subset B(z,\epsilon)$, we conclude that for each $i$, $1\leq i\leq t$, there exists some $j\neq i$, such that 
	 $f^r(\overline{z{e_j}})\supset \overline{ze_i}$.
	It follows that for some $m,k\ 1\leq m,k\leq t$, $f^{mr}(\overline{ze_k})\supset \overline{ze_k}$.
{In particular $e_k\in f^{mr}(\overline{ze_k})$, which is a contradiction.}
	Indeed, $T$ is degenerate.

\end{proof}

\begin{theorem}\label{thm:periods}
	Let $T$ be a tree, $f\colon T\to T$ its continuous selfmap and $(C(T),\tilde{f})$ the induced system on the hyperspace of subcontinua of $T$.
	 Let $A\in C(T)$ be the periodic point of $\tilde{f}$ containing an element $z$ of $T$ such that $f(z)=z$. 
	Then the period of $A$ is a divisor of $\lcm\{1,2,...,\vert \End(T)\vert\}$.
\end{theorem}

\begin{proof}
	Assume there is some $x\in A$ such that $\omega_f(x)$ is a periodic orbit.
	Surely, if $p$ is period of that periodic orbit, we can find some nonempty subset $B\subset \omega_f(x)$ such that $f^{np}(A)\cap\omega_f(x)=B$ for all $n\geq 0$.
	But we can say something similar in the case of the other types of $\omega$-limit sets as well, as shown below.
	
	Let us suppose there is $x\in A$ such that $\omega_f(x)$ is contained in a maximal solenoid $\omega$.
	Lemma~\ref{lem:solenoid-components} provides us with some cycle of trees $K$ for which we can say the following:
	if $p$ is a period of $K$, then for each component $C$ of $K$ either for all $n\geq 0$ it holds $C\subset f^{np}(A)$ or $C\cap f^{np} (A)\cap\omega=\emptyset$ for all $n\geq 0$.
	
	Similarly, by Lemma~\ref{lem:basicset-sub} and Corollary~\ref{cor:basicset-sub}, if $\omega_f(x)$ is an infinite subset of a basic set $\omega$ for some $x\in A$ and $K_1,K_2,...,K_r$ are the trees described in the lemma, we can say that for each $i$, 
	$1\leq i\leq r$, either $f^{nr}(A)\cap K_i\cap \omega=\emptyset$ for all $n\geq 0$, $K_i\subset f^{nr}(A)$ for all $n\geq 0$ or $f^{nr}(A)\cap K_i\cap \omega$ is a single periodic point for all $n\geq 0$.
	The {intersection being a singleton instead of just a finite set in the} latter of the three possibilities is a special consequence of the fact that the underlying space is a tree and not a graph, that $K_i$ is connected and of the fact that $A$ contains a fixed point.	
	
	Of course, if we went over all $\omega$-limit sets of all the points in $A$, then just the observation above generally would not provide us with a mutual iterate of $f$ under which $A$ would have the described properties for all the $\omega$-limit sets.
	Still, using the fact that $A$ contains a fixed point we will be able to find such mutual iterate and have it uniformly bounded for all subtrees of $T$ possessing a fixed point so let us show how we do it.
	
	Take any $x\in A$ and let $K$ be equal to $\omega_f(x)$ if it is a periodic orbit and otherwise let $K$ be the cycle of graphs described in the Lemma~\ref{lem:solenoid-components} or~\ref{lem:basicset-sub}. 	
	Moreover, if $\omega_f(x)$ is an infinite subset of a basic set $\omega$, suppose that for subtrees $K_i$ described in Lemma~\ref{lem:basicset-sub} it holds that $f^{nr}(A)\cap K_i\cap \omega=\emptyset$ for all $n\geq 0$ or $K_i\subset f^{nr}(A)$ for all $n\geq 0$ as we will be dealing with the other ones later. 	
	The component $C$ of $K$ is said to be \emph{closest to $z$} if there is an arc $L$ in $T$ having $z$ as one endpoint and the other endpoint in $C$ and not intersecting $K\setminus C$.
	The number of components of $K$ which are closest to $z$ is less than or equal to the number of endpoints in $T$.
	Let us numerate the components of $K$ as $C_1,C_2,...,C_p$ in a way that for some $1\leq t\leq p$, $C_1,C_2,...,C_t$ are the ones which are the closest to $z$.
	Let us, for $1\leq i\leq t$, denote by $L_i$ the unique arc in $T$  with one of its endpoints being $z$ and the other one being the element of $C_i\cap\omega_f(x)$ such that {$L_i\cap K\cap \omega_f(x)$} is a singleton.
	This is possible since $C_i$ is closest to $z$ and $C_i\cap\omega_f(x)$ is closed so it has the maximal and minimal element in the ordering of each edge it intersects.
	Note that in the special case of $z\in C_i$ and $i=t=1$, either $z\in\omega_f(x)$ so we allow that $L_1=\{z\}$ or in fact we have $t^\prime\geq 2$ such arcs
	 { on respective sides}
	of $z$, but $t^\prime$ is still not greater than $\vert\End(T)\vert$.
	In the latter case let us denote those arcs as $L_1,L_2,...,L_{t^\prime}$ and replace $t$ by $t^\prime$.
	
	Since $\omega_f(x)\cap A\neq\emptyset$, there is some $1\leq i\leq t$ such that $L_i\subset A$.
	Note that for each $1\leq j\leq t$ there is some $1\leq m\leq t$ such that $f(L_j)\supset L_m$.
	It is easy to see that after 
	{ considering an iterate of $A$, we may}
	 find some $1\leq j,k\leq t$, $0\leq r\leq t$ such that $L_j\subset f^r(A)$ and $f^k(L_j)\supset L_j$.	
	For the sake of convenience, let us, without loss of generality, assume that $j=1$.
	
	Observe that, while $\tilde{f}$ permutes the components of $K$ in a cyclic manner, under the {action} of $\tilde{f}^k$ the set of the components of $K$ breaks into $q$ disjoint cycles of length $p/q$ where $q$ is the greatest common divisor of $p$ and $k$.
	Note that, since $A$ is periodic,  $f^k(L_1)\supset L_1$ and $L_1\subset f^r(A)$ then $f^{nk}(L_1)\subset f^r(A)$ for every $n\geq 0$.
	Indeed, for every $n\geq 0$ and every $m\geq 0$, $f^{nk}(L_1)\subset f^{(n+m)k}(L_1) \subset f^{r+(n+m)k}(A)$.
	But $f^{r}(A)$ is a periodic point of $\tilde{f}^{k}$ so $f^{nk}(L_1)\subset f^r(A)$.
	
	This implies that $C_i\subset f^r(A)$ for all $1\leq i\leq p$ such that $i\equiv 1 \pmod{q}${, because $K$ is the smallest possible cycle of graphs containing $\omega_f(x)$}.
	Note that the fact that $A$ is periodic implies that $f^n(A)$ contains the same number of components of $K$ for all $n\geq 0$.
	Therefore, if $f^r(A)$ contains $p/q$ components of $K$ then we are done, i.e. we found $k\geq 1$ such that, for each $1\leq i\leq p$, $C_i\subset f^{nk}(A)$ for all $n\geq 0$ or $f^{nk}(A)\cap C_i\cap\omega_f(x)=\emptyset$ for all $n\geq 0$. 	
	Also, if for every $0\leq j\leq q-1$ we have that either $C_i\subset A$ for all $1\leq i\leq p$ such that $i\equiv j \pmod{q}$ or $C_i\cap A\cap \omega_f(x)=\emptyset$ for all $1\leq i\leq p$ such that $i\equiv j \pmod{q}$ then we are done.
	By what is observed in earlier this paragraph and for the sake of convenience, we can assume that $r=0$.
	
	Now let us suppose the otherwise, i.e. that there is some $0\leq j\leq q-1$ and $1\leq i_1<i_2\leq p$ such that $i_1\equiv i_2\equiv j \pmod{q}$ and $C_{i_1}\subset A$ while $C_{i_2}\cap A\cap \omega_f(x)=\emptyset$.
	Then there are some $1\leq j_1<j_2\leq p$ (namely $j_1=i_1+ q-j+1 \mod{p}$, $j_2=i_2+ q-j+1\mod{p}$) such that  $j_1\equiv j_2\equiv 1 \pmod{q}$, $C_{j_1}\subset f^{q-j+1}(A)$ and $C_{j_2}\cap f^{q-j+1}(A)\cap \omega_f(x)=\emptyset$.
	Then there is some $m\geq 0$ such that $f^{mk}(C_{j_1})=C_1\subset f^{mk+q-j+1}(A)$.
	Since $f^{mk+q-j+1}(A)$ contains $z$ and $C_{1}$, it necessarily also contains $L_1$ and again we get that $f^{nk}(C_1)\subset f^{mk+q-j+1}(A)$ for all $n\geq 0$ and hence $C_i\subset f^{mk+q-j+1}(A)$ for all $1\leq i\leq p$ such that $i\equiv 1 \pmod{q}$.
	 {Consequently,} $C_{i}\subset f^{mk+q}(A)$ for all $1\leq i\leq p$ such that $i\equiv j \pmod{q}$.
	But $q$ is the divisor of $k$ so if we act with $f^q$ {on these sets} $k/q-1$ times we get that $C_{i}\subset f^{(m+1)k}(A)$ for all $1\leq i\leq p$ such that $i\equiv j \pmod{q}$	. 
	 The set of all such components $C_i$ is fixed under $f^k$ so in particular $C_{i_2}\subset f^{nk}(A)$ for all $n\geq m+1$ and therefore $C_{i_2}\subset A$, a contradiction.	
	
	Now denote by $s$ the lowest common multiple of the elements of a set\linebreak $\{1,2,...,\vert \End(T)\setminus \{z\}\vert\}$.
	It is clear that for each $x\in A$, such that $\omega_f(x)$ is not a basic set, there is a cycle of (possibly degenerate) graphs $K$ containing $\omega_f(x)$ such that for each component $C$ of $K$, either $C\subset f^{ns}(A)$ for all $n\geq 0$ or $C\cap f^{ns}(A)\cap\omega_f(x)=\emptyset$.
	However, note that, since this property works for periodic points and each maximal basic set $\omega=E(K,f)$ such that, for each component $C$ of $K$ {we have } $C\subset A$ or $C\cap A\cap \omega=\emptyset$,
	{we can extend it onto other $\omega$-limit sets contained in $\omega$ as well. Hence, if}
	 $\omega=E(K,f)$ is maximal basic set containing $\omega_f(x)$ for some $x\in A$, then for each component $C$ of $K$ either $C\subset f^{ns}(A)$ for all $n\geq 0$ or $C\cap f^{ns}(A)\cap\omega=\emptyset$ for all $n\geq 0$ or there is a periodic point $y$ such that $C\cap f^{ns}(A)=\{y\}$ for all $n\geq 0$.
	Since $\omega_f(x)\cap A\neq\emptyset$ for all $x\in A$, we in fact proved the following:
	if we denote $\omega^\prime=\bigcup_{x\in A}\omega_{f^s}(x)$ then $\omega^\prime\subset A$ and for each  $x\in A$ there is some cycle of (possibly degenerate) graphs $K$ for $f^s$ with components possibly having one-point intersection such that $K\subset A$.
	
	Having that said, there is a finite set $Y=\{y_1,y_2,...,y_m\}$ of all the points of $\omega^\prime$ which are 
	{farthest} from $z$ in a sense that for any $y\in\omega^\prime\setminus Y$, $\overline{zy}\cap Y=\emptyset$.
	For each $1\leq i\leq m$ there is some cycle of graphs $K_{y_i}\subset A$ for $f^s$, described above, with component $C_{y_i}\ni y_i$.
	Let $M$ be the minimal subtree of $T$ containing $\bigcup_{i=1}^m C_{y_i}$.
	It follows that, for all $1\leq i\leq m$, $K_{y_i}\subset M$.
	Since $f^s(K_{y_i})=K_{y_i}$, $f^s(M)$ also contains $C_{y_i}$ for all $1\leq i\leq m$.
	By minimality, $M\subset f^s(M)$.
Hence, $N=\overline{\bigcup_{n\geq 0}f^{ns}(M)}$ is a subtree of $A$, strongly invariant under $f^s$.	
	
	Consider a factor system $(T/_{N},f^s/_{N}),$ where $T/_{N}$ is the factor space and\linebreak $f^s/_{N}\colon T/_{N}\to T/_{N} $ is given by $f^s/_{N}\circ\pi_{N}=\pi_{N}\circ f^s,$ where $\pi_{N}\colon T\to T/_{N}$ is the canonical projection.
	Note that $\pi_N(A)$ is periodic element of $\C(T/_{N})$ with the period under $f^s$ equal to the one of $A$ in $\C(T)$.
	Also, for every $x\in \pi_N(A)$, $\omega_{f^s/_{N}}(x)=\pi_N(N)$ and $\pi_N(N)$ is obviously a singleton.
	But now we are done, since by Lemma~\ref{lem:omega_singleton}, 
	$\pi_N(A)$ is singleton, i.e. $N=A$ and therefore $f^s$ fixes $A$.
\end{proof}

\begin{corollary}\label{cor:periods}
Denote $m=\lcm\{1,2,...,\vert \End(T)\vert\}$ and let $P_1$ and $P_2$ be two intersecting periodic points of $(C(T),\tilde{f})$ of periods $p_1$ and $p_2$, respectively.
If $p_1>mp_2$ then $P_1\subset P_2$.
\end{corollary}

\begin{proof}
	Suppose on the contrary, i.e. that 
    {$P_1\setminus P_2 \neq \emptyset$}.
	Since $f^{p_2}(P_2)=P_2$, we can define a factor system $(T/_{P_2},f^{p_2}/_{P_2}),$ where $T/_{P_2}$ is the factor space and $f^{p_2}/_{P_2}\colon T/_{P_2}\to T/_{P_2} $ is given by $f^{p_2}/_{P_2}\circ\pi_{P_2}=\pi_{P_2}\circ f^{p_2},$ where $\pi_{P_2}\colon T\to T/_{P_2}$ is the canonical projection.
	Note that $\lcm\{1,2,...,\vert \End(T)/_{P_2}\vert\}$ is a divisor of $m$ and $\pi_{P_2}(P_2)$ is a fixed point under $f^{p_2}/_{P_2}$ contained in $\pi_{P_2}(P_1)$.
	But the period of $\pi_{P_2}(P_1)$ under $\tilde{f^{p_2}}/_{P_2}$ is equal to the period of $P_1$ under $\tilde{f}^{p_2}$ and is strictly greater than $m$, a contradiction.
\end{proof}

{\begin{example}
Let $k_1=3$, $k_2=4$, $k_3=2$ and $k_4=5$
Let $A^i_j=[i,i+1/2]\times\{j\}$ for $i=1,2,3,4$ and $j=0,\ldots, k_i-1$ be a family of intervals in $\mathbb{R}^2$.
Let $X=\bigcup_{i,j}A^i_j$ and define $f((x,j))=(x,j+1 \mod k_i)$ for $x\in A_j^i$. Clearly $f$ is continuous.
Now, let $Y$ be a star obtained from $X$ by collapsing all points in $\mathbb{Z}^2\cap X$ to a point, let $\phi\colon X\to Y$ be the associated quotient map and let $g\colon Y\to Y$ be the map induced by $f$ (which is well defined and continuous, since $f(X\cap \mathbb{Z}^2)=X\cap \mathbb{Z}^2$. Now, let $A=\phi(A_0^1)\cup \phi(A_0^2)$ and $B=\phi(A_0^1)\cup \phi(A_0^3)\cup\phi(A_0^4)$ be two periodic continua for $g$. Then period of $A$ is $12$ and period of $B$ is $30$ while $|\End(Y)|=14$. While $A\cap B$ is a continuum with nonempty interior
    
\end{example}
}
\section{center}\label{sec:center}

{In this section we provide a characterization of Birkhoff center $\C(\tilde{f})$ of the induced system on the hyperspace $\mathcal{C}(T)$ of tree subcontinua, that is the set $\C(\tilde{f}):=\overline{\Rec(\tilde{f})}$.}


\begin{theorem}\label{thm:center}
	Let $f$ be a continuous selfmap of a topological tree $T$ and let $A\in \C(\tilde{f})$ be nondegenerate.
	Then $A$ is either periodic or asymptotically degenerate.
	Moreover, every two nondegenerate but asymptotically degenerate elements of $C(\tilde{f})$ are disjoint.
	In particular, for every nondegenerate and asymptotically degenerate $A\in C(\tilde{f})$ and for every $0\leq m<k$, $f^m(A)\cap f^k(A)=\emptyset$.
\end{theorem}

\begin{proof}
	We begin by proving the first assertion of the theorem.	
	By Theorem~\ref{thm:main} (see also \cite[Theorem 4.1]{Matviichuk}) 
	$A$ is either asymptotically degenerate or asymptotically periodic.
	If $A$ is asymptotically degenerate we are done so suppose it is asymptotically periodic and let us prove that it is periodic. 
	 Let $\{A_n\}_{n\geq 1}$ be a sequence of recurrent points of $\tilde{f}$ such that $\lim_n A_n=A$.
	We can assume that $\{A_n\}_{n\geq 1}$ is nondegenerate for all $n\geq 0$ so, by Theorem~\ref{thm:rec-cont}, $\{A_n\}_{n\geq 1}$ is periodic for all $n\geq 0$.  
	 Moreover, for all $n\geq 1$, denote by $p_n$ the period of $A_n$.
	 If the set $\{p_n\colon n\geq 1\}$ is bounded from above by some $M\geq 1$ then, by 
	 continuity of $\tilde{f}^{M!}$, $f^{M!}(A)=\lim_n f^{M!}(A_n)=\lim_n A_n=A$ so $A$ is indeed periodic in this case.
	 
	 Now let us assume that the set $\{p_n\colon n\geq 1\}$ is unbounded. 
	  We will show that this case is impossible.
	 Denote $m=\lcm\{1,2,...,\vert \End(T)\vert\}$.
	 Without loss of generality we can assume that, for every $n\geq 1$, $p_{n+1}>mp_n$.
	 Since $A$ is nondegenerate, there is some $n_0\geq 1$ such that $A_{n_1}\cap A_{n_2}\neq \emptyset$ for all $n_1,n_2\geq n_0$.
	 {Removing a few first elements if necessary, we may} assume that $n_0=1$.
By Corollary~\ref{cor:periods}, $\{A_n\}_{n\geq 1}$ is a decreasing sequence. 
	Let $p$ be the cardinality of $\omega_{\tilde{f}}(A)$ and denote $F_A=\lim_n f^{np}(A)$. 
	We claim that $A=F_A$.

First let us show that $A\cap F_A\neq\emptyset$.
	Suppose	on the contrary, that $F_A\cap A=\emptyset$.
	Then there is some $r\geq 1$ such that $A_r\cap F_A=\emptyset$.
{But $\{A_n\}_{n\geq 1}$ is decreasing sequence and then $\lim_n A_n=A$ gives $A\subset A_r$. Therefore, $f^{npp_r}(A)\subset A_r$ for all $n\geq 0$ and hence $A_r\cap F_A\neq \emptyset$, a contradiction.}
	Indeed, $A\cap F_A\neq\emptyset$.
{By the previous argument, we also know that  $A_n\cap F_A\neq \emptyset$ for every $n\geq 0$.}
	But $F_A$ is periodic with {period $p$ and $p_n>mp$} for all but at most finitely many $n\geq 0$.
	It follows that $A_n\subset F_A$ for all but at most finitely many $n\geq 0$ which yields that $A\subset F_A.$

If $A\neq F_A$, there is some $t\geq 1$ such that $F_A\setminus A_t\neq\emptyset$.
	But $A\subset A_t$, $f^{p_t}(A_t)=A_t$ while $\lim_n f^{np_tp}(A)=F_A$, which is a contradiction.	
	Therefore, $A=F_A$ so $A\subset A_n\subset F_A=A$ for all but at most finitely many $n$.
	
	
	 {Next we are going to prove, that if we} take any two distinct nondegenerate and asymptotically degenerate $A_1,A_2\in C(\tilde{f})$ such that $A_1\cap A_2\neq \emptyset$ 
	then $A_1=A_2$.
 So fix two sets $A_1,A_2$ as above. 
	By Corollary~\ref{cor:periods}, 
	there are nested decreasing sequences $\{A^1_n\}_{n\geq 0}$ and $\{A^2_n\}_{n\geq 0}$ of periodic points of $\tilde{f}$ with strictly increasing periods such that $\lim_n A^1_n=A_1$ and $\lim_n A^2_n=A_2$.
	For every $n\geq 0$, $A^1_n$ intersects $A^2_k$ for all $k\geq 0$.
	But since the periods of $A^2_k$ are strictly increasing, by Corollary~\ref{cor:periods}, for all but at most finitely many $k\geq 0$, $A^2_k\subset A^1_n$.
	Therefore, $A_2\subset A^1_n$ for all $n\geq 0$.
	Hence $A_2\subset A_1$.
Analogously we get that $A_2\subset A_1$ and we are done.	

	 Finally, fix some nondgenerate but asymptotically degenerate {$A\in \C(\tilde{f})$} and suppose that $f^m(A)\cap f^k(A)\neq\emptyset$ for some $m<k$.
	By previous arguments, there is a decreasing sequence $\{A_n\}_{n\geq 0}$ of periodic points of $\tilde{f}$ such that $\cap_n A_n=A$.
	But then also $\{f^m(A_n)\}_{n\geq 0}$ and $\{f^k(A_n)\}_{n\geq 0}$ make decreasing sequences of periodic points of $\tilde{f}$ converging towards $f^m(A)$ and $f^k(A)$, respectively.
	But each element of one of those sequences intersects all the elements of the other one so, just like in the previous paragraph, we obtain that $f^m(A)=f^k(A)$.
	However, $A$ is asymptotically degenerate, so $f^m(A)$ is a singleton fixed by $\tilde{f}^{k-m}$.
	But then the elements of the sequence $\{f^m(A_n)\}_{n\geq 0}$ are periodic points of $\tilde{f}$ with unbounded sequence of periods containing the same periodic point, which contradicts the result stated in Theorem~\ref{thm:periods}.
	This proves the last assertion.
\end{proof}

	Now we are ready to fully characterize $\C(\tilde{f})$ in case of a tree map $f$.
	Let $T$ be a topological tree and let $\mathcal{K}=\{K_n\}_{n\geq 1}$ be a sequence of cycles of trees generating a solenoid in $T$.
	Denote by $\mathcal{S}(\mathcal{K})$ the set of all nondegenerate connected components of $Q=\cap_n K_n$.
	Now let $\mathcal{S}$ be the union of all $\mathcal{S}(\mathcal{K})$ over all generating sequences $\mathcal{K}$. 

\begin{theorem}\label{thm:center-charact}
	Let $f$ be a continuous selfmap of a topological tree $T$. 
	Then
\begin{equation}\label{eqn:center}
\C(\tilde{f})=\{\{x\}\colon x\in\C(f)\}\cup\Per(\tilde{f})\cup\mathcal{S}.
\end{equation}
\end{theorem}

\begin{proof}

	Let $A\in\C(\tilde{f})\setminus\Per(\tilde{f})$ be nondegenerate.
	We claim that $A$ is a wandering tree which is a connected component of the intersection of a generating sequence of some solenoid.
	
	Since, by Theorem~\ref{thm:center}, $f^n(A)\cap A=\emptyset$ for all $n>0$, obviously $A\cap\Per(f)=\emptyset$.
	We also claim that $A$ does not intersect any basic set.
	To prove this assertion, let us suppose on the contrary, that there is some maximal basic set $\omega$ such that $A\cap\omega\neq\emptyset$.  
{Repeating argument from the proof of Theorem~\ref{thm:center}, we obtain }
a decreasing sequence $\{A_n\}_{n\geq 1}$ of some periodic trees with unbounded sequence of periods $\{p_n\}_{n\geq 1}$ such that $A=\cap_n A_n$.
	Therefore, for every $n\geq 0$, $A_n$ intersects $\omega$.
	Let $K_1,K_2,...,K_m$ be trees described in Lemma~\ref{lem:basicset-sub}, whose union contains $\omega$.
	If $\Int A_n$ intersects $\omega\cap K_i$ for some $1\leq i\leq m$ and infinitely many $n\geq 0$, then all of such $A_n$ contain $K_i$ and hence $K_i\subset A$, which is impossible as $A$ is asymptotically degenerate.
	Therefore, all but at most finitely many $A_n$ intersect $\omega\cap K_i$ only with the boundary for every $1\leq i\leq m$.
	This, by Corollary~\ref{cor:basicset-sub} implies that all but at most finitely many elements of $\{A_n\}_{n\geq 1}$ contain some periodic point from some $K_i$.
	But, since $\{A_n\}_{n\geq 1}$ forms a nested sequence and since those $A_n$ intersect $\omega\cap K_i$ only at the boundary, there can be at most finitely many such periodic points so there are infinitely many elements of $\{A_n\}_{n\geq 1}$ containing same periodic point $z\in\omega$.
	But then $z\in A$, leading to a contradiction.	

	By the fixed-point property, every $A_n$ intersects $\Per(f)$ and hence $A\cap\overline{\Per(f)}\neq\emptyset$.
	This, together with the above observation, implies that $A$ intersects some solenoidal $\omega$-limit set.
	It also intersects only one maximal $\omega$-limit set since otherwise it would not be asymptotically degenerate, as distinct maximal solenoidal $\omega$-limit sets are compact and pairwise disjoint and therefore have positive distance.	
	
	Now take some point $x\in A\cap \overline{\Per(f)}$ from a maximal solenoidal $\omega$-limit set $Q_{\max}$.
	There is a nested (decreasing) sequence $\{M_n\}_{n\geq 1}$ of trees with strictly increasing and unbounded sequence of periods $\{r_n\}_{n\geq 1}$ such that $Q_{\max}\subset Q=\cap_n M_n$.
	For each $n\geq 1$ let us denote the components of $M_n$ by $M^i_n$ for $1\leq i\leq r_n$.
	 For a point $x$ introduced above there is a sequence $\{k_n\}_{n\geq 1}$, with $1\leq k_n\leq r_n$ for all $n\geq 0$, such that $x\in \cap_n M_n^{k_n}$.
	 We want to prove that $A=\cap_n M_n^{k_n}$.
	 
	 To that end, firstly fix some $n\geq 1$ and denote $t=\lcm\{1,2,...,\vert \End(T)\vert\}$.
	 Since $x\in M_n^{k_n}$ and $x\in A\subset A_k$ for all $k\geq 1$ and since $p_k>tr_n$ for all but at most finitely many $k$, by Corollary~\ref{cor:periods} it follows that $A_k\subset M_n^{k_n}$ for all but at most finitely many $k\geq 1$.
	 In particular, $A\subset M_n^{n_k}$.
	 But $n\geq 1$ was arbitrary so $A\subset \cap_n M_n^{k_n}$.
	 Now fix some arbitrary $k\geq 1$. 
	 Again, $x\in M_n^{k_n}$ for all $n\geq 1$ and $x\in A\subset A_k$ so we have that $ M_n^{k_n}\cap A_k\neq\emptyset$ for all $n\geq 1$.
	 But $r_n>tp_k$ for all but at most finitely many $n\geq 1$.
	 Therefore $M_n^{k_n}\subset A_k$ for all but at most finitely many $n$ and, in particular, $\cap_n M_n^{k_n}\subset A_k$.
	 But this holds for any $k\geq 0$ and therefore $\cap_n M_n^{k_n}\subset A$ and we are done.
	
	This proves the inclusion $\subset$ in~\eqref{eqn:center}.
	The opposite inclusion is trivial.
\end{proof}

\section{Almost equicontinuity}\label{sec:equicontinuity}

	Let $T$ be a tree, $E\subset T$ be an edge of $T$ and $X\subset E$ be its subset.
	We will say that a point $y\in \Int E$ is {two-sided} 
	accumulation point of $X$ if {every} one-sided neighborhood of $y$ intersects $X$.

	First we prove this easy lemma.
\begin{lemma}\label{lem_accumulation_pts}
	Let $X$ be an uncountable subset of an edge of $T$.
	Then all but at most countably many points of $X$ are its two-sided accumulation points.	
\end{lemma}	

\begin{proof}
	Suppose on the contrary, i.e. that there are uncountably many points in $X$ which are not its two-sided accumulation points.
	Let $Y\subset X$ be the set of such points which are cut points of $T$.
	For each such $y$, there is an open free arc $L_y$, contiguous to $y$, such that $L_y\cap X=\emptyset$.
	Without loss of generality, we can assume that $Y$ contains uncountable subset $Y^\prime$ such that, for all $y\in Y^\prime$, $y$ is the minimal point of $\overline{L_y}$ in the ordering of the edge. 
	Then the set $\{L_y\colon y\in Y^\prime\}$ consists of uncountably many pairwise disjoint open arcs contained in an edge of $T$, which is a contradiction.
\end{proof}


\begin{proof}[Proof of Theorem~\ref{thm:almost_equi}]
	Since the set of all the points of equicontinuity is $G_\delta$ for 
    {every} compact dynamical system, it remains to prove that the set $\mathcal{E}$ of points of equicontinuity of $(C(T),\tilde{f})$ is dense in $C(T)$.
	To this end, let us take any open set $\mathcal{U}$ in $C(T)$ and let us prove that $\mathcal{U}\cap\mathcal{E}\neq\emptyset$.
	Note that there is some open set $\mathcal{U}^\prime\subset \mathcal{U}$ not containing singletons and such that $J\cap\End(T)=\emptyset$ for {all $J\in \mathcal{U}^\prime$ and $J\cap \Br(T)=\Int J\cap \Br(T)$ for all $J\in \mathcal{U}^\prime$.
This allows us to assume without loss of generality that $\mathcal{U}=\mathcal{U}^\prime$. }

{If there is some asymptotically degenerate $A\in \mathcal{U}$ then let us pick some $C\in \mathcal{U}$, $C\subset\Int A$.
	We claim that $C\in\mathcal{E}$.
	Fix any $\epsilon>0$.
	There is some $n_0\geq 0$ such that $\diam f^n(A)<\epsilon$ for all $n\geq n_0$.
	There exists $\delta>0$ such that $N(C,\delta)\subset A$ (recall properties of $\mathcal{U}^\prime$) and $\dist(f^n(D),f^n(C))<\epsilon$ for all $D\in B_{\dist_H}(C,\delta)$ and all $0\leq n\leq n_0$.
	This proves the claim and completes the proof for this case.

	Next, let us consider the remaining case, which by Theorem~\ref{thm:main} allows us to assume that all the elements of $\mathcal{U}$ are asymptotically periodic.
	Fix a set} $A\in \mathcal{U}$ and let $p$ be the cardinality of $\omega_{\tilde{f}}(A)$.
	Denote $F_A=\lim_n f^{np}(A)$.
	Then $F_A$ is a subtree of $T$, fixed under $f^p$.
	Denote $m=\lcm\{1,2,...,\vert \End(T)\vert\}$.	
	
	Note that for any $B\supset A$, the limit $\lim_n f^{nmp}(B)$ exists and it is a superset of $F_A$. 
	 Indeed, $\omega_{\tilde{f^p}}(B)$ is a periodic orbit of some period $r$.
	  Hence, $\lim_n f^{npr}(B)$ exists and equals some subtree $R$ of $T$ which contains $\lim_n f^{npr}(A)=F_A$.
	  But then $r$, as the period of $R$ under $f^p$, is a divisor of $m$ {by Theorem~\ref{thm:periods}}. 
	  Therefore, the limit $\lim_n f^{nmp}(B)$ exists and it is a superset of $F_A$.
	
	Now we can choose some sequence $\{A_n\}_{\geq 0}$ in $\mathcal{U}$ such that $A_0=A$ and $A_{n}\subset\Int A_{n+1}$.
	For each $n\geq 0$ we denote $F_{A_n}=\lim_k f^{kmp}(A_n)$.
	Note that each $F_{A_n}$ is fixed under $f^{mp}$ and also $F_{A_n}\subset F_{A_{n+1}}$ for all $n\geq 0$.
	Since there are finitely many branching points in $T$, there are some $0\leq n_1<n_2$ such that $F_{A_{n_1}}$ and $F_{A_{n_2}}$ intersect interiors of the same edges of $T$. 	   
	$\mathcal{U}^\prime=\{D\in C(T)\colon A_{n_1}\subset\Int D\text{ and } D\subset\Int A_{n_2}\}$ is nonempty open subset of $\mathcal{U}$ so we can {again}, without loss of generality, assume that $\mathcal{U}=\mathcal{U}^\prime$.
	Since $A_{n_1}\subset \Int A_{n_2}$, there is some $\delta>0$ such that $N(A_{n_1},\delta)\subset A_{n_2}$.
	For each $0<\gamma<\delta$, we define $B_\gamma=\overline{N(A_{n_1},\gamma)}\in \mathcal{U}$.
	Note that $\lim_n f^{nmp}(B_\gamma)$ exists for all $0<\gamma<\delta$ and let us denote that limit by $F_\gamma$.
	Note that, for each $0<\alpha<\gamma <\delta$ we have $B_\alpha\subset B_\gamma$ and $F_\alpha\subset F_\gamma$. 	
	 Now we consider the family $\{F_\gamma\colon \gamma\in(0,\delta)\}$.
	 {Decreasing $\delta$ when necessary, we may require that }for each edge $E$ of $T$ and for each $\alpha,\gamma\in(0,\delta)$, $E\subset F_\alpha$ if and only if $E\subset F_\gamma$, {and similarly we may require that} $E\cap F_\alpha=\emptyset$ if and only if $E\cap F_\gamma=\emptyset$.	
		 
	Therefore, we can numerate all the edges of $T$ by $E_1,E_2,...,E_r$, in a way that for some $t$, $1\leq t\leq r$, and for all $\gamma\in(0,\delta)$, $\partial F_\gamma\cap E_i\neq\emptyset$ for all $1\leq i\leq t$ and $\partial F_\gamma\cap E_i=\emptyset$ for all $t<i\leq r$.
	
	After all these preparations we are finally in position to pick an element of $\mathcal{E}$ in $\mathcal{U}$.
	Consider the set $Y_1=\{\partial F_\gamma\cap E_1\colon \gamma\in(0,\delta)\}$.
	If this set is uncountable then, by Lemma~\ref{lem_accumulation_pts}, there is uncountable set $X_1\subset (0,\delta)$ such that for all $\gamma\in X_1$, the unique element of $\partial F_\gamma\cap E_1$ is two-sided accumulation point of $Y_1$.
	On the other hand, if $Y_1$ is at most countable then there is some uncountable set $X_1\subset (0,\delta)$ and a point $y_1\in Y_1$ such that $\partial F_\gamma\cap E_1=\{y_1\}$ for all $\gamma\in X_1$.
	Now we proceed inductively. 
	For each $k,\ 2\leq k\leq t$ we construct the set $Y_k=\{\partial F_\gamma\cap E_k\colon {\gamma\in X_{k-1}}\}$.
	If $Y_k$ is uncountable, by Lemma~\ref{lem_accumulation_pts}, there is uncountable set $X_k\subset X_{k-1}$ such that for all $\gamma\in X_k$, the unique element of $\partial F_\gamma\cap E_k$ is two-sided accumulation point of $Y_k$.
	In the case when $Y_k$ is at most countable, there is some uncountable set $X_k\subset X_{k-1}$ and a point $y_k\in Y_k$ such that $\partial F_\gamma\cap E_k=\{y_k\}$ for all $\gamma\in X_k$.
 
	We will prove that, for every $\gamma\in X_t$, $B_\gamma\in\mathcal{E}$, so it is enough to pick one of them to complete the proof.
	To that end, let us take some $\gamma\in X_t$ and some $\epsilon>0$.
	Since $\tilde{f}$ is uniformly continuous, there is some $\kappa>0$ such that, for all $K,L\in C(T)$, $\dist_H(K,L)<\kappa$ implies $\dist_H(f^n(K),f^n(L))<\epsilon$ for all $0\leq n\leq mp$.
	By construction, there are some $\alpha_1,\alpha_2$, $0<\alpha_1<\gamma<\alpha_2<\delta$ such that $\dist_H(F_{\alpha_i},F_\gamma)<\kappa/2$ for $1\leq i\leq 2$. 	
	There is some $n_0\geq 0$ such that for all {$n\geq n_0$}, $\dist_H(F_{\alpha_1},f^{nmp}(B_{\alpha_1}))<\kappa/2$ and $\dist_H(F_{\alpha_2},f^{nmp}(B_{\alpha_2}))<\kappa/2$.
	There is also some $\zeta>0$ such that for all $C\in  B_{\dist_H}(B_\gamma, \zeta)$, $B_{\alpha_1}\subset C\subset B_{\alpha_2}$ and,
	for all $0\leq n\leq n_0$, $\dist_H(f^{nmp}(C),f^{nmp}(B_\gamma))<\kappa$.
	Indeed, for  all $C\in  B_{\dist_H}(B_\gamma, \zeta)$ and for all $n\geq 0$, $\dist_H(f^{nmp}(C),f^{nmp}(B_\gamma))<\kappa$.
	Therefore, for all $n\geq 0$, $\dist_H(f^{n}(C),f^{n}(B_\gamma))<\epsilon$. 
	This proves that $B_\gamma\in\mathcal{E}$.
	The theorem is proved.
		
\end{proof}

\section*{Acknowledgements}
{The authors are grateful to the referee of the manuscript for very careful reading and several valuable comments that helped improve exposition in the paper.}

Research of D. Jeli\'c was supported by NextGenerationEU foundation, project: IP-UNIST-44 (ITPEM).
This paper was partially written during Domagoj Jeli\'c’s visit to Charles University in Prague, funded by the Croatian Science Foundation under the Outbound Mobility of Research Assistants program (Call ID: MOBDOK-2023-2114).

\bibliographystyle{plain}
\bibliography{bibAsymptoticalBehavior2}	
	
%
\end{document}